\documentclass[11pt,reqno]{amsart}

\usepackage[utf8]{inputenc}
\usepackage[T1]{fontenc}
\usepackage{amsmath,amssymb,amsthm,amsfonts}
\usepackage{mathtools}
\usepackage{geometry}
\usepackage{bm}
\usepackage{enumitem}
\usepackage[normalem]{ulem}
\usepackage{xcolor}

\numberwithin{equation}{section}

\usepackage{citeref}
 
\setlist[enumerate,2]{
  label=(\alph*),
  ref=\theenumi(\alph*)
}

\usepackage[hyphens]{url}
\usepackage[pdftex,breaklinks]{hyperref}

\newtheorem{theorem}{Theorem}[section]
\newtheorem{proposition}[theorem]{Proposition}
\newtheorem{lemma}[theorem]{Lemma}

\newtheorem{corollary}[theorem]{Corollary}

\newtheorem{maintheorem}{Theorem}
\newtheorem{maincorollary}[maintheorem]{Corollary}   

\theoremstyle{definition}

\newtheorem{remark}[theorem]{Remark}

\newcommand{\corr}{\operatorname{Corr}}

\newcounter{AssumpA}

\newenvironment{assump}{\em
  \begin{enumerate}[
    label={\normalfont\textbf{(A\arabic*)}}, 
    ref={\normalfont(A\arabic*)},
    font=\normalfont
  ]
    \setcounter{enumi}{\value{AssumpA}}
}{
    \setcounter{AssumpA}{\value{enumi}}
  \end{enumerate}
}

\newcommand{\diam}{\mathrm{diam}}
\newcommand{\Lip}{\mathrm{Lip}}
\newcommand{\qand}{\quad\text{and}\quad}

\DeclareMathOperator*{\lip}{Lip}

\title[Decay of Correlations for Partially Hyperbolic Skew-Products]{Decay of Correlations for\\ Partially Hyperbolic Skew-Products 
}
\date{}

\thanks{JFA was partially supported by The Royal Society Wolfson Visiting Fellowship RSWVF$\backslash$25$\backslash$R1$\backslash$1006 at the Department of Mathematical Sciences, Loughborough University hosted by WB; and by CMUP, member of LASI, which is financed by national funds through FCT -- Funda\c{c}\~{a}o para a Ci\^encia e a Tecnologia, I.P., under project UID/00144/2025 (DOI: \url{https://doi.org/10.54499/UID/00144/2025}) and PTDC/MAT-PUR/4048/2021.}  

 \author[J. F. Alves]{Jos\'{e} F. Alves}
\address{Jos\'{e} F. Alves\\ Departamento de Matem\'{a}tica\\ Faculdade de Ci\^encias da Universidade do Porto\\ Rua do Campo Alegre 687\\ 4169-007 Porto\\ Portugal}
\email{jfalves@fc.up.pt} \urladdr{http://www.fc.up.pt/cmup/jfalves}

\author[W. Bahsoun]{Wael Bahsoun}
\address{Wael Bahsoun, Department of Mathematical Sciences, Loughborough University, Loughborough, Leicestershire, LE11 3TU, UK}
\email{W.Bahsoun@lboro.ac.uk}
 \urladdr{https://www.lboro.ac.uk/departments/maths/staff/wael-bahsoun}


\begin{document}

\maketitle

\begin{abstract} We develop a technique based on sectional transfer operators and the regularity of their fixed points to study rates of correlation decay for \emph{partially hyperbolic} skew-products.  To illustrate the range of applicability of this technique, we apply it to a family of Viana like maps with \emph{non-uniformly contracting} fibres, a family of skew-product maps with \emph{non-uniformly expanding} base dynamics and to a generalised family of  \emph{heterochaos} baker maps. The latter partially hyperbolic systems have recently been an  active topic of research due to their connection with the Dyck shift. \end{abstract}

\tableofcontents

\section*{Introduction}

Partially hyperbolic systems are a fundamental class of dynamical systems, combining hyperbolic behaviour with centre dynamics that are neither uniformly contracting nor expanding. Their ability to capture interactions across multiple timescales makes them central to modern dynamical systems theory. Despite their importance, their ergodic theory remains far less developed than that of uniformly hyperbolic systems, with many fundamental questions still open.

A main objective in the study of partially hyperbolic systems is the existence of physical measures and the statistical laws they encode. In \cite{Tsujii00} Tusjii proved existence of finitely many physical measures for partially hyperbolic surface endomorphisms with one-dimensional strongly unstable subbundle. In the setting of mostly expanding centre direction, the pioneering work of Alves, Bonatti and Viana~\cite{ABV00} established the existence of Sinai-Ruelle-Bowen measures through the introduction of hyperbolic times. This approach was subsequently developed into a general framework based on Gibbs--Markov--Young structures \cite{Y98,Y99} and inducing schemes, yielding precise information on rates of mixing and statistical limit laws; see, for instance, Alves and Pinheiro~\cite{AP08}, and Alves and Li~\cite{AL15}. The complementary setting of mostly contracting centre direction was initiated by Bonatti and Viana~\cite{BV00}. Although the existence and finiteness of physical measures are now well understood in this context, considerably less is known regarding quantitative statistical properties such as rates of decay of correlations. Existing results treat only exponentially mixing smooth diffeomorphsims, see Dolgopyat \cite{D00}, or  often rely on additional geometric assumptions or special structures; see, for example, de Castro~\cite{Ca01}.

A particularly important class of partially hyperbolic systems is provided by skew-product maps where the statistical behaviour of the fibre dynamics is strongly influenced by the properties of the base transformation. Skew-products with uniformly contracting fibres have attracted considerable attention. In particular, Araújo, Galatolo and Pacifico~\cite{AGP14} established decay of correlations and logarithm laws for systems with uniformly contracting fibres, while Galatolo~\cite{G18} obtained quantitative statistical stability and convergence to equilibrium for broad classes of partially hyperbolic skew-products. Related developments may also be found in the work of Kloeckner~\cite{K21}.
Another important class of partially hyperbolic skew-products has a uniformly expanding base with fibre dynamics given by the identity plus rotation. Such systems were studied by  Gouëzel~\cite{G09}, Butterley and Eslami~\cite{BE} among others, and recent generalisations by Castorrini and Liverani~\cite{CL22} and Tsujii \cite{T23}. They arise naturally in deterministic fast--slow systems, where the base dynamics acts as a rapidly evolving environment for the fibre dynamics. This point of view has led to a substantial literature on averaging and statistical limit laws; see Dolgopyat~\cite{D04},  De Simoi and Liverani~\cite{DeL18}, Korepanov, Kosloff and Melbourne~\cite{KKM17} and the recent work by De Simoi, Fernando and Fleming-Vázquez \cite{DFF25}. These works illustrate the richness of skew-product structures and highlight the need for techniques capable of treating systems beyond the above settings.

The purpose of the present paper is to develop a general framework to study rates of correlations decay for partially hyperbolic skew-products whose fibre dynamics is \emph{merely contracting on average}. More precisely, we assume that a suitable logarithmic average of the fibre Lipschitz constants is negative. This condition may be viewed as an averaged negativity condition on the fibre Lyapunov exponent and is substantially weaker than uniform contraction. Our approach is based on the introduction of a sectional transfer operator acting on measurable families of probability measures on the fibres. We prove that this operator admits a unique invariant section and that this section is Hölder continuous. The Hölder regularity of the invariant section established here is the key ingredient in transferring statistical properties from the base dynamics to the full skew-product. Thus, the work provides a flexible approach to obtain rates of correlation decay for a wide range of partially hyperbolic skew products without the need for constructing Young-Towers.
The approach also treats polynomially and exponentially mixing systems that fall outside the classical transfer operator framework which require the system to be defined via a stable or a central cone field (see the books \cite{B18, DKL21} and the recent works \cite{BL22, CL22}).

Our first main result establishes the existence and uniqueness of a physical measure whose basin is completely determined by the basin of the absolutely continuous invariant measure of the base transformation. Our second main result shows that if the base transformation exhibits decay of correlations at a rate \(r(n)\), then the skew-product inherits the same asymptotic rate, up to an exponentially small error term. 

To illustrate the applicability of the abstract theory, we consider three families of examples. The first consists of heterochaotic baker maps, introduced in~\cite{STY21,TY25}, which exhibit coexistence of invariant sets with different unstable dimensions and they are related to the Dyck shift \cite{TY25}. Recent work has established, in a conservative setting, exponential and polynomial mixing properties for certain parameter regimes~\cite{T25,TT25}. We show that our methods apply beyond the conservative setting and in parameter regions where the invariant measure is not explicitly known. The second family is a smooth skew-product that may be viewed as a nonuniformly contracting counterpart of the classical Viana maps~\cite{V97}. Whereas Viana maps are characterised by nonuniformly expanding fibres over an expanding base, our examples possess fibres that contract only on average despite the presence of regions exhibiting strong expansion. Finally, we consider skew-products over intermittent Pomeau--Manneville maps~\cite{LSV99,CHM10}. Since these maps exhibit only polynomial decay of correlations, our results yield examples of partially hyperbolic skew-products with polynomial rates of mixing, demonstrating that the theory developed here is not restricted to rapidly mixing base dynamics. The examples considered in this paper illustrate the flexibility of the abstract framework developed here and show that average fibre contraction provides a natural setting in which one can establish existence, uniqueness, and quantitative statistical properties of physical measures for a broad class of partially hyperbolic skew-products.

The paper is organised as follows. In Section~2 we introduce the sectional transfer operator and establish the existence, uniqueness, and Hölder regularity of its invariant section, together with the existence and uniqueness of the associated physical measure. In Section~3 we prove our abstract result on decay of correlations. The remaining sections are devoted to applications: heterochaotic baker maps, smooth nonuniformly contracting skew-products, and skew-products over intermittent maps.
\smallskip
\paragraph{\em Acknowledgements.}
The authors would like to thank Marcelo Viana for valuable communications.

\part{General Theory}\label{part.general}

In this part, we develop the abstract framework underlying the paper. We begin by introducing a sectional transfer operator acting on measurable families of probability measures and establish its fundamental properties under an average fibre contraction assumption. We then prove the existence and uniqueness of an invariant section and show that it inherits Hölder regularity from the base dynamics and the fibre maps, despite the fact that contraction holds only in an averaged logarithmic sense. Next, we use this invariant section to construct a unique physical measure for the skew-product and show that its basin is completely determined by the basin of the invariant measure of the base transformation. Finally, we establish our abstract result on decay of correlations, proving that the statistical properties of the base are transferred to the full skew-product up to an exponentially small error term. Together, these results provide a robust framework for partially hyperbolic skew-products with average fibre contraction, establishing existence and uniqueness of physical measures and quantitative rates of decay of correlations under assumptions inherited from the base dynamics.

\section{Setup and Main Results}\label{se.setup}

Let $T:M\to M$ be a discrete-time dynamical system on a manifold $M$, and let $\mu$ be a $T$-invariant probability measure. The \emph{basin} of $\mu$ is defined by
\[
B(\mu)
=
\left\{
x\in M :
\frac{1}{n}\sum_{j=0}^{n-1}\delta_{T^j(x)}
\stackrel{w^*}{\longrightarrow}
\mu
\text{ as } n\to\infty
\right\}.
\]
Thus, $B(\mu)$ consists of those points whose empirical measures converge to $\mu$ in the weak-$^*$ topology. The measure $\mu$ is called a \emph{physical measure} if
$
 B(\mu)  $
has positive  Lebesgue measure on $M$.

\subsection{Skew-Product Maps} 
Let \( \Omega , X \) be compact Riemannian manifolds (possibly with boundary), and let \(d_\Omega,d_X\) denote the respective Riemannian distances.  We endow $\Omega\times X$ with the product Riemannian structure induced by the Riemannian metrics on $\Omega$ and $X$ and use $d$ to denote the corresponding distance on $\Omega\times X$. Let $\pi: \Omega \times X \to \Omega$ be  the canonical  projection and $\mathrm{diam}(X)$ denote the diameter of $X$.

We consider a skew-product map
\(T : \Omega \times X \to \Omega \times X\), defined by
\begin{equation} \label{eq.skewmap}
T(\omega,x) = \big(f(\omega),\, g(\omega,x)\big),
\end{equation}
for \((\omega,x) \in \Omega \times X\).
We shall refer to  $f:\Omega\to\Omega$ as the \emph{base map}.
For each \(\omega \in \Omega\), we denote by \(g_{\omega} : X \to X\) the \emph{fibre map} given by
\[
g_{\omega}(x) = g(\omega,x),\qquad x\in X.
\]

We now introduce the standing assumptions that will be used throughout the remainder of the paper. These assumptions provide the abstract framework for all subsequent results. We begin by fixing some notation.
Given metric spaces $(Y,d_Y)$ and $(Z,d_Z)$ and a Lipschitz continuous function $\varphi:Y\to Z$, let
\[
\Lip(\varphi)
:=
\sup_{x\neq y}
\frac{d_Z(\varphi(x),\varphi(y))}{d_Y(x,y)}.
\]
Given $q\in(0,1)$ and a $q$-Hölder continuous function $\varphi:Y\to Z$ define \[
[\varphi]_q
:=
\sup_{x\neq y}
\frac{d_Z(\varphi(x),\varphi(y))}{d_Y(x,y)^q}
\]
Our first two assumptions impose  uniform Lipschitz bounds on the fibre maps and some regularity of the base dynamics.

\begin{assump}
\item \label{as.fibre} There exists $L>0$ such that, for every $\omega\in\Omega$, the map
$g_\omega:X\to X$ is Lipschitz continuous with
$$\frac1L \le  \lip(g_\omega) \le L.$$
\end{assump} 

\begin{assump}
\item \label{as.base}
The following hold for the base map:
\begin{enumerate}
\em
\item \em \label{as.base1}
there exist pairwise disjoint open sets
$
\Omega_1,\ldots,\Omega_N\subset \Omega
$
  such that
$
m\!\left(
\Omega\setminus \cup_{i=1}^{N}\Omega_i
\right)=0.
$
Moreover, for every \(i=1,\ldots,N\), the restriction
$
f\vert_{ {\Omega}_i}$ has an extension to the closure~$\overline{\Omega}_i$ which is a $C^1$ bijection onto $ \Omega
$. Denoting by
$
\theta_i:\Omega\to\overline{\Omega}_i
$
the corresponding inverse branch, we assume that 
$$\lip(\theta_i)\le 1.$$
\em
  \item \em  \label{as.base2}
  \(f\) has an invariant probability measure \(\eta\) that is equivalent  to~\(m\), with density \(\rho \). 
\end{enumerate}
\end{assump} 

Set for each $\omega\in\Omega$ and $i=1,\dots,N$
 \begin{equation}\label{eq.weight}
p_i( \omega)=\frac{\rho(\theta_i(\omega))}{|\det D f(\theta_i(\omega))| \, \rho(\omega)} .
\end{equation}

Our next assumption is a mild   contraction on average condition for the fibre maps.

\begin{assump}
    \item\label{as.mean}  
   There exists $\lambda < 0$ such that for $\eta$-almost every $\omega\in \Omega$
    \begin{equation*}
    \sum_{i=1}^N p_i(\omega)\log \left(\Lip(g_{\theta_i(\omega)})\right) \le \lambda.
    \end{equation*}
\end{assump}

The assumptions above already suffice to establish the existence of a fixed point for the sectional transfer operator that we introduce below. To obtain the regularity of this fixed point, we impose two further conditions. 
Given $x\in X$ and $1\le i\le N$, define $g_{x,i}:\Omega\to X$ by
$$g_{x,i}(\omega)=g_{\theta_i(\omega)}(x)= g(\theta_i(\omega),x),\qquad \omega\in\Omega.
$$

\begin{assump}
\item \label{as.holder}
There exists \(q,L>0\) such that for all \(i=1,\dots,N\),
\begin{enumerate}
\em
\item\label{as.holder1} $[p_i]_q\le L$;
 
\item\label{as.holder2}
$
  \lip(g_{x,i})\le L
$, for all $x\in X$.
\end{enumerate}
\end{assump}
Note that if \ref{as.holder1} holds for some value of $q$, then it also holds for every smaller value of $q$.

\subsection{Sectional Transfer Operator} The framework above enables the introduction of an operator that provides a natural setting for studying the evolution of the \emph{sample measures} under the skew-product dynamics. We begin by defining the space on which this operator acts.

Let $\mathcal M_1(X)$ denote the set of Borel probability measures on~$X$. This set endowed with the  Wasserstein-1 distance $W_1$, the space $\mathcal M_1(X)$ is a complete metric space. Since $X$ is assumed to be compact, $\mathcal M_1(X)$ is also bounded with respect to~$W_1$; see Subsection~\ref{eq.wasserstein} for details about   Wasserstein metrics.
We define the space of probability-valued measurable sections over $\Omega$ by
\[
\mathbf P
=
\Bigl\{
\boldsymbol{\nu}:\Omega\to\mathcal M_1(X)
\;\Big|\;
\boldsymbol{\nu}\ \text{is measurable}
\Bigr\},
\]
where sections that agree $\eta$-almost everywhere are identified. For
$\boldsymbol{\nu}\in\mathbf P$, we write $\nu_\omega=\boldsymbol{\nu}(\omega)$ and denote the corresponding family of measures by
$\{\nu_\omega\}_{\omega\in\Omega}$.
We equip $\mathbf P$ with the metric
\begin{equation}\label{eq.dp}
d_{\mathbf P}(\boldsymbol{\mu},\boldsymbol{\nu})
=
\operatorname*{ess\,sup}_{\omega\in\Omega}
W_1(\mu_\omega,\nu_\omega).
\end{equation}
Since $(\mathcal{M}_1(X), W_1)$ is bounded, this metric is well defined, and 
$\mathbf{P}$ is in fact a complete metric space.
For $\boldsymbol{\nu}\in\mathbf P$, the \emph{sectional transfer operator}
$\mathcal K:\mathbf P\to\mathbf P$ is defined, for $\eta$-almost every
$\omega\in\Omega$, by
\begin{equation}\label{eq.kalpha}
(\mathcal K\boldsymbol{\nu})_\omega
=
\sum_{i=1}^{N}
p_i(\omega)\,
\bigl(g_{\theta_i(\omega)}\bigr)_*
\nu_{\theta_i(\omega)}.
\end{equation}
Using that $\rho$ is a fixed point of the transfer operator associated with the base dynamics $f$, we readily deduce that
\begin{equation}\label{eq:sumpi}
\sum_{i=1}^N p_i(\omega)=1
\end{equation}
for $\eta$-almost every $\omega\in\Omega$. Therefore,
$
(\mathcal K\boldsymbol{\nu})_\omega \in \mathcal M_1(X)
$
for $\eta$-almost every $\omega\in\Omega$.
By the measurability of all functions involved in its definition, we have that $\mathcal K$ defines a map from~$\mathbf P$ into itself. Moreover, a fixed point of~$\mathcal K$ in $\mathbf P$ provides the sample measures of a $T$-invariant measure; see \cite[Corollary 2.5]{AB26}. Conversely, we show in Lemma~\ref{le.fixedeq} that the sample measures of any $T$-invariant measure whose marginal on $\Omega$ is~$\eta$ is necessarily a fixed point of $\mathcal K$ in the set $\mathbf P$.

\subsection{Main Results}
Our first main result establishes existence, uniqueness, and Hölder regularity of the fixed point of \(\mathcal K\). This regularity property will subsequently play a key role in the proof of decay of correlations for the skew-product map.

\begin{maintheorem}\label{th.mainA}
Assume that    \ref{as.fibre}--\ref{as.holder} hold. Then 
\begin{enumerate}
 \item  $\mathcal K $  admits a unique fixed point $\bm\nu  \in\mathbf P$;
  \item  the section
$
\bm\nu=\{\nu_\omega\}_{\omega\in\Omega}
$
is the disintegration of a $T$-invariant probability measure $\mu$   with marginal $\eta$ on $X$ and
 $B(\mu)=B(\eta)\times X;$
\item   there exists $q\in(0,1)$ such that  $\bm\nu :\Omega\to \mathcal M_1(X)$ is $q$-H\"older continuous.
\end{enumerate}

\end{maintheorem}

The relation between the basins of $\mu$ and $\eta$ in item (2) above  yield the following.
 
 \begin{maincorollary}\label{co.mainA}
 Under the conditions of Theorem~\ref{th.mainA},
\begin{enumerate}
\item  
$\mu$ is a physical measure for $T$ 
$\iff$ $\eta$ is a physical measure for $f$;
\item
 $T$ has a unique  physical measure 
$\iff$   $f$ has a unique  physical measure.
\end{enumerate}
\end{maincorollary}

To state our next main result on decay of correlations for the skew-product map, we first introduce the regularity classes of observables and the corresponding norms that will be used.
For a compact metric space $(Y,d_Y)$ and $q\in(0,1)$, let
$$
\lip(Y)
=
\left\{
\varphi:Y\to\mathbb R
:
\lip(\varphi)<\infty
\right\}
$$
and
\[
C^q(Y)
=
\left\{
\varphi:Y\to\mathbb R
:
[\varphi]_q<\infty
\right\}.
\]
We equip $\lip(Y)$ with the norm
\[
\|\varphi\|_{\lip}
:=
\|\varphi\|_\infty+\Lip(\varphi).
\]
and    $C^q(Y)$ with the norm
\[
\|\varphi\|_q
:=
\|\varphi\|_\infty+[\varphi]_q.
\]
For observables $\phi,\psi:\Omega\to\mathbb R$, we define the \emph{correlation function} with respect to the $f$-invariant probability measure $\eta$ by
\[
\corr_\eta(\phi,\psi)
:=
\left|
\int \phi\,\psi\,d\eta
-
\int \phi\,d\eta
\int \psi\,d\eta
\right|.
\]
We assume that the base dynamics exhibits decay of correlations at a rate controlled by a function~$r$.

\begin{assump}
\item \label{as.dc}
There exists a nonincreasing function
$
r:[0,\infty)\to[0,\infty)
$
such that for every $q>0$, every bounded measurable
$\psi:\Omega\to\mathbb R$
and every
$\phi\in C^q(\Omega)$,
\[
\corr_\eta(\phi,\psi\circ f^n)
\le
r(n)\,
\|\psi\|_\infty\,
\|\phi\|_{q}.
\]
\end{assump}

\begin{maintheorem}\label{th.DC}
Assume that \ref{as.fibre}--\ref{as.dc} hold and  $\mu$ is the probability measure given by Theorem~\ref{th.mainA}.
Given $\kappa\in(0,1)$, there exist $C>0$ and $s\in(0,1)$ such that for all
$\Phi,\Psi\in\Lip(\Omega\times X)$ and all $n\ge1$,
\[
\left|
\int \Phi\cdot\Psi\circ T^n\,d\mu
-
\int \Phi\,d\mu
\int \Psi\,d\mu
\right|
\le
C\bigl(\|\Phi\|_{\lip}+\|\Psi\|_{\lip}\bigr)
\bigl[r(\kappa n)+s^n\bigr].
\]
\end{maintheorem}

\section{Invariant Section}

The main goal of this section is to prove Theorem~\ref{th.mainA}. We shall also obtain some useful results for the proof of Theorem~\ref{th.DC} in Section~\ref{se.DC}.

\subsection{Wasserstein Metrics}\label{eq.wasserstein}

A \emph{coupling} between two Borel probability measures \(\mu \) and \(\nu \) on \(X\) is a Borel probability measure \(\gamma \) on \(X\times X\) whose marginals are \(\mu \) and \(\nu \). That is,
\[
\gamma (A\times X)=\mu (A),
\qquad
\gamma (X\times B)=\nu (B)
\]
for every Borel sets \(A,B\subset X\).
Intuitively, \(\gamma \) describes how to transport mass from \(\mu \) to \(\nu \): if \((x,y)\) is sampled according to \(\gamma \), then \(x\) has law \(\mu \) and \(y\) has law \(\nu \).
The \emph{Wasserstein-$q$   distance} is defined for $q\ge1$  by minimizing the average transport cost over all such couplings:
\begin{equation}\label{eq.wq}
W_q(\mu ,\nu )
=
\left(
\inf_{\gamma\in\Gamma(\mu,\nu )}
\int_{X\times X} d_X(x,y)^q\,d\gamma(x,y)
\right)^{1/q},
\end{equation}
where \(\Gamma(\mu,\nu)\) denotes the set of all couplings of \(\mu \) and \(\nu \). 
An \emph{optimal coupling} is simply a coupling \(\gamma \in\Gamma(\mu,\nu )\) which realizes the infimum in~\eqref{eq.wq}.
Such an optimal coupling always exists whenever~\(X\) is a compact space; see e.g. \cite[Theorem 2.3.2]{FG21}. 

Although the expression in~\eqref{eq.wq} does not define a distance for $q\in (0,1)$, the expression in~\eqref{eq.wq}   is still well-defined in this range, and this is precisely the regime in which it will be useful for us to consider \(W_q\).
For notational simplicity, we will consider for $q\in (0,1)$
\begin{equation}\label{eq.wqtilde}
\widetilde{W}_q(\mu ,\nu )
= W_q(\mu ,\nu )^q =
\inf_{\gamma\in\Gamma(\mu,\nu )}
\int_{X\times X} d_X(x,y)^q\,d\gamma(x,y).
\end{equation}

\begin{lemma}\label{le.w1wq}
 For any $q\in (0,1)$ and any probability measures $\mu, \nu \in \mathcal{M}_1(X)$,  we have
 \begin{equation*}
\widetilde W_q(\mu, \nu) \le W_1(\mu, \nu)^q
\qquad\text{and}\qquad
W_1(\mu, \nu) \le \diam(X)^{1-q} \widetilde{W}_q(\mu, \nu).
\end{equation*}
\end{lemma}

\begin{proof}
We first prove that $W_q(\mu,\nu) \le W_1(\mu,\nu)$. For any coupling $\gamma \in \Gamma(\mu,\nu)$, since $t \mapsto t^q$ is concave on $[0,\infty)$ for $0<q<1$, Jensen's inequality yields
\[
\int_{X \times X} d_X(x,y)^q \, d\gamma(x,y)
\le \left( \int_{X \times X} d_X(x,y)\, d\gamma(x,y) \right)^q.
\]
Taking the infimum over $\gamma \in \Gamma(\mu,\nu)$ gives

\[
\widetilde W_q(\mu,\nu) \le W_1(\mu,\nu)^q.
\]
We now prove that $W_1(\mu,\nu) \le \diam(X)^{1-q} \widetilde W_q(\mu,\nu)$. For any coupling $\gamma \in \Gamma(\mu,\nu)$, since $d_X(x,y)\le \diam(X)$ and $1-q>0$, we have
\[
d_X(x,y) = d_X(x,y)^{1-q} d_X(x,y)^q \le \diam(X)^{1-q} d_X(x,y)^q.
\]
Integrating with respect to $\gamma$ yields
\[
\int_{X \times X} d_X(x,y)\, d\gamma(x,y)
\le \diam(X)^{1-q} \int_{X \times X} d_X(x,y)^q \, d\gamma(x,y).
\]
Taking the infimum over all couplings $\gamma \in \Gamma(\mu,\nu)$ gives
\[
W_1(\mu,\nu) \le   \diam(X)^{1-q}\,\widetilde{W}_q(\mu,\nu).
\]
This completes the proof.
\end{proof}

The next result is standard for $q \ge 1$; for completeness, we show that it   holds for  all $q > 0$.

\begin{lemma}\label{le.contlip}
Let \(h:X\to X\) be \(L\)-Lipschitz. Then, for any $q>0$ and any  \(\mu,\nu\in\mathcal M_1(X)\),
\[
\widetilde W_q(h_*\mu,h_*\nu)\le L^q\,\widetilde W_q(\mu,\nu).
\]
\end{lemma}
\begin{proof}
Let \(\gamma\) be an optimal  coupling of \(\mu\) and \(\nu\). Define a new measure
\[
\widetilde\gamma=(h\times h)_*\gamma
\]
on \(X\times X\). 
The first marginal of \(\widetilde\gamma\) is \(h_*\mu\), and the second marginal is \(h_*\nu\), so \(\widetilde\gamma\) is a coupling of \(h_*\mu\) and \(h_*\nu\).
Therefore,
\[
\widetilde W_q(h_*\mu,h_*\nu) 
\le
\int d_X(u,v)^q\,d\widetilde\gamma(u,v).
\]
By definition of push-forward,
\[
\int d_X(u,v)^q\,d\widetilde\gamma(u,v)
=
\int d_X(h(x),h(y))^q\,d\gamma(x,y).
\]
Since \(h\) is \(L\)-Lipschitz,
\[
d_X(h(x),h(y))^q
\le
L^q d(x,y)^q.
\]
Hence
\[
\widetilde W_q(h_*\mu,h_*\nu) 
\le
L^q\int d_X(x,y)^q\,d\gamma(x,y).
\]
Since  \(\gamma\) has been chosen  an optimal coupling between \(\mu\) and \(\nu\), we have
\[
\int d_X(x,y)^q\,d\gamma(x,y)=\widetilde W_q(\mu,\nu).
\]
Then
\[
\widetilde W_q(h_*\mu,h_*\nu)
\le
L^q \widetilde W_q(\mu,\nu),
\]
which gives the desired conclusion.
\end{proof}

The next result gives an estimate for the $\widetilde W_q$ image of  convex sums of probability measures. This will be useful in the proof of Proposition~\ref{le.qholder}.

\begin{lemma}\label{le.wassum}
Assume that  $\mu_1,\mu_1',\dots,\mu_N,\mu_N'$ are measures in $ \mathcal M_1(X)$ and $w_1,w_1',\dots ,w_N,w_N'\ge 0$ are such that $$\sum_{i=1}^N w_i=1\qand \sum_{i=1}^N w_i'=1.$$ Then
 \begin{equation*}
\widetilde{W}_q\left(\sum_{i=1}^Nw_i\mu_i,\sum_{i=1}^Nw_i'\mu_i'\right)
\le
\sum_{i=1}^N \min\{w_i,w_i'\}\widetilde{W}_q(\mu_i,\mu_i')
+ \diam(X)^q \sum_{i=1}^N |w_i-w_i'|.
\end{equation*}

\end{lemma}
\begin{proof}
First, we construct explicitly a  coupling between
\[
\mu = \sum_{i=1}^Nw_i\mu_i
\qquad
\mu' = \sum_{i=1}^Nw_i'\mu_i'
\]
that will help us to obtain an upper bound for $\widetilde{W}_q(\mu,\mu')$ in terms of $\widetilde{W}_q(\mu_i,\mu'_i)$.
Define
\[
a_i = \min\{w_i, w_i'\}, 
\quad
\varepsilon_i = w_i - a_i,
\quad
\varepsilon_i' = w_i' - a_i.
\]
Then
\[
\sum_{i=1}^N a_i + \sum_{i=1}^N \varepsilon_i = 1,
\quad
\sum_{i=1}^N a_i + \sum_{i=1}^N \varepsilon_i' = 1,
\]
and
\[
\sum_{i=1}^N \varepsilon_i = \sum_{i=1}^N \varepsilon_i' = \frac12 \sum_{i=1}^N |w_i-w_i'|.
\]
Let $\gamma_i$ be optimal couplings between $\mu_i$ and $\mu_i'$.
Given  an arbitrary probability measure $\sigma$ on $X$, define the measure
\[
\gamma =
\sum_{i=1}^N a_i \gamma_i
\;+\;
\sum_{i=1}^N \varepsilon_i (\mu_i \otimes \sigma)
\;+\;
\sum_{i=1}^N \varepsilon_i' (\sigma \otimes \mu_i'),
\]
with the symbol  $\otimes$ being used to denote  the product measure.
Then $\gamma$ is a coupling of $\mu$ and $\mu'$.
Therefore 
\[
\widetilde{W}_q(\mu,\mu')
\le
\int d(x,y)^q\, d\gamma(x,y).
\]
Splitting according to the definition of $\gamma$,
\begin{align*}
\int d(x,y)^q\, d\gamma
&=
\sum_{i=1}^N a_i \int d(x,y)^q\, d\gamma_i
\\
&\quad +
\sum_{i=1}^N \varepsilon_i \int d(x,y)^q\, d(\mu_i \otimes \sigma)
\\
&\quad +
\sum_{i=1}^N \varepsilon_i' \int d(x,y)^q\, d(\sigma \otimes \mu_i').
\end{align*}
Therefore,
\[
\int d(x,y)^q\, d(\mu_i \otimes \sigma) \le \diam(X)^q,
\quad
\int d(x,y)^q\, d(\sigma \otimes \mu_i') \le \diam(X)^q.
\]
Hence
 \begin{equation}\label{eq.wsum}
\widetilde{W}_q(\mu,\mu')
\le
\sum_{i=1}^N a_i \widetilde{W}_q(\mu_i,\mu_i')
+ \diam(X)^q \sum_{i=1}^N |w_i-w_i'|,
\end{equation}
thus yielding the desired conclusion.
\end{proof}

 \subsection{Mean Contraction}

In this subsection  we show that, under assumptions~\ref{as.fibre}--\ref{as.holder},  the sectional transfer operator $\mathcal K$ has a unique fixed point in the  metric space  $\mathbf P$.

 \begin{lemma} \label{le.meanq}
For $q>0$ sufficiently small, there exists $\tau\in(0,1)$ such that for $\eta$-almost every $\omega\in\Omega$,
\[
\sum_{i=1}^N p_i(\omega)\,\Lip(g_{\theta_i(\omega)})^q \le \tau.
\]
\end{lemma}

\begin{proof}
Set for $i=1,\dots,N$
$$
a_i(\omega)=\Lip(g_{\theta_i(\omega)}).
$$
By~\ref{as.fibre}, there exists a constant 
$
M >0
$
be such that, for all $i=1,\dots,N$ and $\eta$-almost all $\omega\in\Omega$,
\[
|\log a_i(\omega)| \le M.
\]
For $x\in[-M,M]$ and $q> 0$, Taylor's formula gives
\[
e^{qx}
=
1 + qx + R(q,x),\quad |R(q,x)| \le C q^2,
\]
with a constant $C>0$ (depending only on $M$). Applying this to $x=\log a_i(\omega)$ yields
\[
a_i(\omega)^q
=
1 + q\log a_i(\omega) + R_i(q,\omega),
\quad |R_i(q,\omega)| \le C q^2.
\]
Define
\[
F_q(\omega):=\sum_{i=1}^N p_i(\omega)a_i(\omega)^q.
\]
Using the expansion above for $a_i(\omega)^q$,
\[
F_q(\omega)
=
\sum_{i=1}^N p_i(\omega)
+
q\sum_{i=1}^N p_i(\omega)\log a_i(\omega)
+
\sum_{i=1}^N p_i(\omega)R_i(q,\omega).
\]
Since \eqref{eq:sumpi}
holds, we get
\[
F_q(\omega)
=
1 + q\sum_{i=1}^N p_i(\omega)\log a_i(\omega) + E(q,\omega),
\]
where
\[
|E(q,\omega)|
= \sum_{i=1}^N p_i(\omega)|R_i(q,\omega)|
\le C q^2 \sum_{i=1}^N p_i(\omega)
= C q^2.
\]
Using~\ref{as.mean}, we obtain
\[
F_q(\omega)
\le 1 + q\lambda + C q^2.
\]
Choose $q>0$ such that
\[
Cq \le \frac{|\lambda|}{2}.
\]
Then
\[
F_q(\omega)
\le 1 + \frac{\lambda}{2}q.
\]
Set
\[
\tau= 1 + \frac{\lambda}{2}q.
\]
Since $\lambda<0$, we have $\tau\in(0,1)$. Therefore,
\[
\sum_{i=1}^N p_i(\omega)\,a_i(\omega)^q \le \tau
,
\]
for $\eta$-almost every $\omega\in\Omega$.
\end{proof}
 
 \begin{remark}\label{re.qtau}
By Lemma~\ref{le.meanq}, for all sufficiently small $q\in(0,1)$ the desired conclusion holds. 
Choosing $q$ sufficiently small, we may assume that \ref{as.holder1} also holds for this value of $q$. 
Let $\tau\in(0,1)$ be the corresponding value given by Lemma~\ref{le.meanq}.
\end{remark}

\begin{proposition}\label{pr.Kcontraction}
There exists
\(C>0\)  
such that for all \(\bm\mu,\bm\nu\in\mathbf P\) and \(n\ge 1\)
\[
d_{\mathbf{P}}(\mathcal K^n\bm\mu,\mathcal K^n\bm\nu)
\le
C\tau^n d_{\mathbf{P}}(\bm\mu,\bm\nu)^q.
\]
\end{proposition}

\begin{proof}
For \(\bm\mu,\bm\nu\in\mathbf P\), define
\[
\Delta_q(\bm\mu,\bm\nu)
=
\sup_{\omega\in\Omega}
\widetilde{W}_q(\mu_\omega,\nu_\omega).
\]
We do not claim that $\Delta_q$ defines a metric, since $\widetilde W_q$ itself fails to be a metric for  $q\in(0,1)$.
We first show that $\mathcal K$ is a ``contraction'' with respect to $\Delta_q$.
Fix \(\omega\in\Omega\). Write
\[
(\mathcal K\bm\mu)_\omega=\sum_{i=1}^N p_i(\omega)\,\xi_i,
\qquad
(\mathcal K\bm\nu)_\omega=\sum_{i=1}^N p_i(\omega)\,\zeta_i,
\]
where
\[
\xi_i=(g_{\theta_i(\omega)})_*\mu_{\theta_i(\omega)},
\qquad
\zeta_i=(g_{\theta_i(\omega)})_*\nu_{\theta_i(\omega)}.
\]
For each \(i=1,\dots,N\), let \(\gamma_i\) be an optimal coupling between \(\zeta_i\) and \(\xi_i\), so that
\[
\widetilde{W}_q(\zeta_i,\xi_i)=\int d_X(x,y)^q\,d\gamma_i(x,y).
\]
Define
\[
\gamma=\sum_{i=1}^N p_i(\omega)\,\gamma_i,
\]
which is a coupling between \((\mathcal K\bm\mu)_\omega\) and \((\mathcal K\bm\nu)_\omega\). Using Lemma~\ref{le.wassum} with $w_i=w_i'=p_i(\omega)$
\[
\widetilde{W}_q((\mathcal K\bm\mu)_\omega,(\mathcal K\bm\nu)_\omega)
\le
\sum_{i=1}^N p_i(\omega)\,\widetilde{W}_q(\zeta_i,\xi_i).
\]
By Lemma~\ref{le.contlip},
\[
\widetilde{W}_q(\xi_i,\zeta_i)
\le
\Lip(g_{\theta_i(\omega)})^q\,
\widetilde{W}_q(\mu_{\theta_i(\omega)},\nu_{\theta_i(\omega)}).
\]
Therefore,
\[
\widetilde W_q((\mathcal K\bm\mu)_\omega,(\mathcal K\bm\nu)_\omega)
\le
\sum_{i=1}^N p_i(\omega)\Lip(g_{\theta_i(\omega)})^q
\widetilde{W}_q(\mu_{\theta_i(\omega)},\nu_{\theta_i(\omega)})
\le
\sum_{i=1}^N p_i(\omega)\Lip(g_{\theta_i(\omega)})^q
\Delta_q(\bm\mu,\bm\nu).
\]
Using Lemma~\ref{le.meanq}, we obtain
\[
\Delta_q(\mathcal K\bm\mu,\mathcal K\bm\nu)
\le
\tau\,\Delta_q(\bm\mu,\bm\nu).
\]
Hence, for all $n\ge 1$,
\begin{equation}\label{eq.detaqcon}
\Delta_q(\mathcal K^n\bm\mu,\mathcal K^n\bm\nu)
\le
\tau^n \Delta_q(\bm\mu,\bm\nu).
\end{equation}
Now, by the second inequality of Lemma~\ref{le.w1wq},
\begin{align}
d_{\mathbf P}(\mathcal K^n\bm\mu,\mathcal K^n\bm\nu)
&=
\sup_{\omega\in\Omega}
W_1((\mathcal K^n\bm\mu)_\omega,(\mathcal K^n\bm\nu)_\omega)\nonumber\\
&\le
\diam(X)^{1-q}
\sup_{\omega\in\Omega}
\widetilde{W}_q((\mathcal K^n\bm\mu)_\omega,(\mathcal K^n\bm\nu)_\omega)\nonumber\\
&=
\diam(X)^{1-q}
\Delta_q(\mathcal K^n\bm\mu,\mathcal K^n\bm\nu).\label{first}
\end{align}
Using \eqref{eq.detaqcon} and the first inequality of Lemma~\ref{le.w1wq}, we get
\begin{align}
\Delta_q(\mathcal K^n\bm\mu,\mathcal K^n\bm\nu)
&\le
\tau^n
\Delta_q(\bm\mu,\bm\nu)\nonumber\\
&\le
\tau^n
\sup_{\omega\in\Omega}
\widetilde W_q(\mu_\omega,\nu_\omega) \nonumber\\
&\le
\tau^n
\sup_{\omega\in\Omega}
W_1(\mu_\omega,\nu_\omega)^q\nonumber\\
&=
\tau^n d_{\mathbf P}(\bm\mu,\bm\nu)^q.\label{second}
\end{align}
Combining \eqref{first} and \eqref{second} yields the conclusion with \(C=\diam(X)^{1-q}\).
\end{proof}



\begin{corollary}\label{co.fixK}
The operator $\mathcal K$ admits a unique fixed point $\bm{\nu} \in \mathbf P$. Moreover,
\[
\mathcal K^n \bm{\mu} \longrightarrow \bm{\nu}, \quad \text{as } n \to \infty,
\]
for every $\bm{\mu} \in \mathbf P$.
\end{corollary}

\begin{proof}
This follows immediately from Proposition~\ref{pr.Kcontraction} and and the fixed point Lemma~\ref{le.fix}.
\end{proof}

\subsection{Hölder Regularity}

We prove that the invariant section inherits the Hölder regularity of the base dynamics and fibre maps. The main difficulty is that contraction holds only in a logarithmic (averaged) sense, so we work at the level of iterates and derive a uniform Hölder bound via a Lasota--Yorke type inequality in the $\widetilde{W}_q$ scale.
Fix  the constants  $q,\tau\in(0,1)$ as in Remark~\ref{re.qtau} 
and define 
\[
\mathbf{P}_{q} = \big\{ \bm{\nu} : \Omega \to \mathcal  M_1(X)  \mid   [\bm{\nu}]_{q} < \infty \big\},
\]
where 
\[
[\bm{\nu}]_{q} =  
 \sup_{\omega \neq \omega'} \frac{\widetilde{W}_q(\nu_\omega, \nu_{\omega'})}{d_\Omega(\omega, \omega')^q}.
\]
Note that by the second inequality of Lemma~\ref{le.w1wq},   every $\bm\nu\in\mathrm P_q$ is a $q$-H\"older continuous map with respect to  the metric $W_1$ on $\mathcal M_1(X)$.

\begin{proposition}
\label{le.qholder}
There exists $C>0$ such that for all $\bm{\nu} \in \mathbf{P}_{q} $ and $n\ge 1$
\[
[(\mathcal K^n \bm{\nu})]_{q}
 \le
\tau^n[\bm{\nu}]_{q}
+
\frac{C}{1-\tau}.\
\]
\end{proposition}

\begin{proof}
Fix $\bm{\nu} \in \mathbf{P}_{q} $ and  $\omega,\omega' \in \Omega$. For each $i=1,\dots N$, define
\[
\mu_i = (g_{\theta_i(\omega)})_* \nu_{\theta_i(\omega)},
\qand
\mu_i' = (g_{\theta_i(\omega')})_* \nu_{\theta_i(\omega')}.
\]
Note that
\[
(\mathcal K\bm\nu)_\omega = \sum_{i=1}^N p_i(\omega)\mu_i
\qand
(\mathcal K\bm\nu)_{\omega'}= \sum_{i=1}^N p_i(\omega')\mu_i'.
\]
By Lemma~\ref{le.wassum}, we have
 \begin{equation}\label{eq.wsum}
\widetilde{W}_q\left((\mathcal K\bm\nu)_\omega,(\mathcal K\bm\nu)_{\omega'}\right)
\le
\sum_{i=1}^N a_i \widetilde{W}_q(\mu_i,\mu_i')
+ \diam(X)^q \sum_{i=1}^N |p_i(\omega)-p_i(\omega')|, 
\end{equation}
where
\[
a_i = \min\{p_i(\omega), p_i(\omega')\}.
\]
Now, we estimate each  $\widetilde{W}_q(\mu_i,\mu_i')$.
Let  $\gamma_i$ be an optimal coupling between $\nu_{\theta_i(\omega)}$ and $\nu_{\theta_i(\omega')}$. Define the pushforward measure
\[
\widetilde{\gamma}_i = (g_{\theta_i(\omega)} \times g_{\theta_i(\omega')})_* \gamma_i.
\]
Then $\widetilde{\gamma}_i$ is a coupling between $\mu_i$ and $\mu_i'$, hence by definition of $\widetilde{W}_q$,
\[
\widetilde{W}_q(\mu_i,\mu_i')
\le
\int d_X(u,v)^q\, d\widetilde{\gamma}_i(u,v).
\]
By the change-of-variables formula for pushforward measures,
\[
\int d_X(u,v)^q\, d\widetilde{\gamma}_i(u,v)
=
\int d_X\bigl(g_{\theta_i(\omega)}(x),\, g_{\theta_i(\omega')}(y)\bigr)^q \, d\gamma_i(x,y).
\]
We now estimate the integrand. For all $x,y \in X$,
\[
d_X\bigl(g_{\theta_i(\omega)}(x), g_{\theta_i(\omega')}(y)\bigr)
\le
d_X\bigl(g_{\theta_i(\omega)}(x), g_{\theta_i(\omega)}(y)\bigr)
+
d_X\bigl(g_{\theta_i(\omega)}(y), g_{\theta_i(\omega')}(y)\bigr).
\]
Using the concavity of the map $z\mapsto z^q$ in $[0,\infty)$ for $q \in (0,1]$, we obtain
\begin{align*}
d_X\bigl(g_{\theta_i(\omega)}(x), g_{\theta_i(\omega')}(y)\bigr)^q
&\le
d_X\bigl(g_{\theta_i(\omega)}(x), g_{\theta_i(\omega)}(y)\bigr)^q
 +
d_X\bigl(g_{\theta_i(\omega)}(y), g_{\theta_i(\omega')}(y)\bigr)^q.
\end{align*}
Integrating with respect to $\gamma_i$ yields
\begin{align*}
\widetilde{W}_q(\mu_i,\mu_i')
&\le
\int d_X\bigl(g_{\theta_i(\omega)}(x), g_{\theta_i(\omega)}(y)\bigr)^q \, d\gamma_i(x,y)
+
\int d_X\bigl(g_{\theta_i(\omega)}(y), g_{\theta_i(\omega')}(y)\bigr)^q \, d\gamma_i(x,y).
\end{align*}
For the first term, since $g_{\theta_i(\omega)}$ is a Lipschitz map by \ref{as.fibre} and   $\gamma_i$ is an optimal coupling between $\nu_{\theta_i(\omega)}$ and $\nu_{\theta_i(\omega')}$,
\begin{align*}
\int d_X\bigl(g_{\theta_i(\omega)}(x), g_{\theta_i(\omega)}(y)\bigr)^q \, d\gamma_i(x,y) &\le
\Lip(g_{\theta_i(\omega)})^q
\int d_X(x,y)^q \, d\gamma_i(x,y)\\
&=
\Lip(g_{\theta_i(\omega)})^q
\widetilde{W}_q(\nu_{\theta_i(\omega)}, \nu_{\theta_i(\omega')}).
\end{align*}
For the second term, since   \ref{as.holder2} holds,
\[
\int d_X\bigl(g_{\theta_i(\omega)}(y), g_{\theta_i(\omega')}(y)\bigr)^q \, d\gamma_i(x,y)
\le
\sup_{y \in X}
d_X\bigl(g_{\theta_i(\omega)}(y), g_{\theta_i(\omega')}(y)\bigr)^q \le
L ^q\, d_\Omega(\omega,\omega')^{  q}.
\]
Combining the above estimates, we conclude that
 \begin{equation}\label{eq.neweq}
\widetilde{W}_q(\mu_i,\mu_i')
\le
\Lip(g_{\theta_i(\omega)})^q
\widetilde{W}_q(\nu_{\theta_i(\omega)}, \nu_{\theta_i(\omega')})
+
L^q d_\Omega(\omega,\omega')^q.
\end{equation}
On the other hand, it follows from the definition of   $[\bm{\nu}]_{q}$ and~\ref{as.base1} that
$\omega,\omega'\in\Omega$,
\[
\widetilde{W}_q(\nu_{\theta_i(\omega)},\nu_{\theta_i(\omega')})
\le
[\bm{\nu}]_{q}\,
d_\Omega\big(\theta_i(\omega),\theta_i(\omega')\big)^q\le  
[\bm{\nu}]_{q}\,
d_\Omega(\omega,\omega')^{  q}.\]
Plugging this into the estimate obtained in \eqref{eq.neweq}, we conclude that
\[
\widetilde{W}_q(\mu_i,\mu_i')
\le
  \Lip(g_{\theta_i(\omega)})^q
[\bm{\nu}]_{q}
d_\Omega(\omega,\omega')^q
+
L^q d_\Omega(\omega,\omega')^q.
\]

Multiplying by $a_i$ and summing over $i$, we obtain
\begin{align*}
\sum_{i=1}^N a_i \widetilde{W}_q(\mu_i,\mu_i')
 \le
[\bm{\nu}]_{q}\,
d_\Omega(\omega,\omega')^q 
\sum_{i=1}^N a_i \Lip(g_{\theta_i(\omega)})^q
  +
L^q  d_\Omega(\omega,\omega')^q  \sum_{i=1}^N a_i.
\end{align*}
Since $\sum_{i=1}^N a_i \le 1$, this gives
\[
\sum_{i=1}^N a_i \widetilde{W}_q(\mu_i,\mu_i')
\le
[\bm{\nu}]_{q}\,
d_\Omega(\omega,\omega')^q
\sum_{i=1}^N a_i \Lip(g_{\theta_i(\omega)})^q
+
L^q  d_\Omega(\omega,\omega')^q .
\]
Next, since $a_i \le p_i(\omega)$, we have
\[
\sum_{i=1}^N a_i \Lip(g_{\theta_i(\omega)})^q
\le
\sum_{i=1}^N p_i(\omega) \Lip(g_{\theta_i(\omega)})^q.
\]
Therefore, using~\eqref{eq.wsum}
\begin{align*}
\widetilde{W}_q\left((\mathcal K\bm\nu)_\omega,(\mathcal K\bm\nu)_{\omega'}\right)
\le &
[\bm{\nu}]_{q}\,
d_\Omega(\omega,\omega')^q
\sum_{i=1}^N p_i(\omega) \Lip(g_{\theta_i(\omega)})^q\\
&
  +
L^q  d_\Omega(\omega,\omega')^q
+
\diam(X)^q  \sum_{i=1}^N |p_i(\omega)-p_i(\omega')|.
\end{align*}
It follows from Lemma~\ref{le.meanq}  that for all $\omega\in\Omega$
\[
  \sum_{i=1}^N p_i(\omega) \Lip(g_{\theta_i(\omega)})^q \le \tau.
\]
Hence,
\[
\widetilde{W}_q\left((\mathcal K\bm\nu)_\omega,(\mathcal K\bm\nu)_{\omega'}\right)
\le
\tau[\bm{\nu}]_{q} d_\Omega(\omega,\omega')^q 
+ L^q  d_\Omega(\omega,\omega')^q+
\diam(X)^q  \sum_{i=1}^N |p_i(\omega)-p_i(\omega')|.
\]
Dividing by $d_\Omega(\omega,\omega')^q$, using~\ref{as.holder1}   and Remark~\ref{re.qtau} yields
\[
[\mathcal K \bm{\nu}]_{q}
\le
\tau[\bm{\nu}]_{q} + C ,
\]
with $C=L^q+\diam(X)^q NL$.
Therefore,
\[
[(\mathcal K^n \bm{\nu})]_{q}
\le
\tau^n[\bm{\nu}]_{q}
+
C\sum_{k=0}^{n-1} \tau^k\le
\tau^n[\bm{\nu}]_{q}
+
\frac{C}{1-\tau}
\]
for all $n\ge1$.
\end{proof}

Elements of $\mathbf P$ are equivalence classes modulo $\eta$. Since we assume in~\ref{as.base}
the measure $\eta$  equivalent to Lebesgue measure on~$\Omega$, every set of full
$\eta$-measure is dense in $\Omega$. Moreover, any $q$-Hölder map from a
dense subset of $\Omega$ into the complete metric space
$(\mathcal M_1(X),W_1)$ admits a unique $q$-Hölder extension to all of
$\Omega$. In what follows, we identify such equivalence classes with
their unique $q$-Hölder representatives.

\begin{corollary}\label{th.holder}
The fixed point of $\mathcal K$ in $\mathbf P$ is a $q$-Hölder
continuous map from $\Omega$ to $\mathcal M_1(X)$.
\end{corollary}

\begin{proof}

Let $\bm\nu^*$ be the fixed point of $\mathcal K$ in $\mathbf P$. Given
$\bm\nu\in\mathbf P_q$, define
\[
\bm\nu^{(n)}=\mathcal K^n\bm\nu , \qquad n\ge0.
\]
By Proposition~\ref{pr.Kcontraction},
\[
d_{\mathbf P}(\bm\nu^{(n)},\bm\nu^*)\to0 .
\]
Hence there exists a set $E\subset\Omega$ with $\eta(\Omega\setminus E)=0$
such that for every $\omega\in E$
\[
W_1(\nu^{(n)}_\omega,\nu^*_\omega)\to0, \qquad\text{as $n\to\infty$}.
\]
By Proposition~\ref{le.qholder}, there exists $C=C(\bm\nu)>0$ such that
\[
[\bm\nu^{(n)}]_q\le C,
\qquad \text{for all }n\ge1.
\]
Using the second inequality of Lemma~\ref{le.w1wq}, we obtain for all
$\omega,\omega'\in\Omega$,
\begin{equation}\label{eq.wineq}
W_1\bigl(\nu^{(n)}_\omega,\nu^{(n)}_{\omega'}\bigr)
\le \diam(X)^{1-q}\widetilde W_q\bigl(\nu^{(n)}_\omega,\nu^{(n)}_{\omega'}\bigr)
\le \diam(X)^{1-q}C\, d_\Omega(\omega,\omega')^q .
\end{equation}
Fix $\omega,\omega'\in E$. By the triangle inequality,
\[
\begin{aligned}
W_1(\nu^*_\omega,\nu^*_{\omega'})
&\le
W_1(\nu^*_\omega,\nu^{(n)}_\omega)
+
W_1(\nu^{(n)}_\omega,\nu^{(n)}_{\omega'})
+
W_1(\nu^{(n)}_{\omega'},\nu^*_{\omega'}).
\end{aligned}
\]
Letting $n\to\infty$, the first and third terms vanish, and it follows that fro all $\omega,\omega'\in E$
\[
W_1(\nu^*_\omega,\nu^*_{\omega'})
\le \diam(X)^{1-q}C\, d_\Omega(\omega,\omega')^q.
\]
Thus $\bm\nu^*$ is represented by a $q$-Hölder map on $E$, and hence,  by a unique $q$-Hölder continuous map on~$\Omega$.
\end{proof}

\subsection{Physical Measures}

Let $(\Omega,\mathcal F,\mathbb P)$ be a probability space and let $X$ be a metric space. Let $\mathcal B$ denote the Borel $\sigma$-algebra of $X$, and let $\mu \in \mathcal M_1(\Omega \times X)$.
We say that $\{\nu_\omega\}_{\omega \in \Omega}$ is a \emph{factorisation} (or a \emph{disintegration}) of $\mu$ with respect to $\mathbb P$ if:
\begin{enumerate}
    \item for every $B \in \mathcal B$, the map $\omega \mapsto \nu_\omega(B)$ is $\mathcal F$-measurable;
    \item for $\mathbb P$-almost every $\omega \in \Omega$, $\nu_\omega$ is a probability measure on $(X,\mathcal B)$;
    \item for every $A \in \mathcal F \otimes \mathcal B$,
    \[
    \mu(A) = \int_\Omega \int_X \mathbf{1}_A(\omega,x)\, \nu_\omega(dx)\, \mathbb P(d\omega).
    \]
\end{enumerate}

It is well known that if $X$ is a Polish space, then a factorisation of $\mu$ exists and is $\mathbb P$-almost everywhere unique; see e.g.~\cite{GS77} or~\cite[Section 1.4]{A98}. After modification on a $\mathbb P$-null set if necessary, we may assume that $\omega \mapsto \nu_\omega$ defines a map from $\Omega$ into $\mathcal M_1(X)$. In the next result, we show that  
this
is indeed a Borel measurable map from $\Omega$ into $\mathcal M_1(X)$, considering  $\mathcal M_1(X)$  equipped with the Wasserstein-1 distance $W_1$.

\begin{lemma}\label{le.uniqueP}
Assume that $X$ is a compact separable metric space. If  $\bm\nu=\{\nu_\omega\}_{\omega \in \Omega}$ is a factorisation of $\mu$ with respect to $\mathbb P$, then $\bm\nu\in\mathbf P$.
\end{lemma}
\begin{proof}
By assumption, for every $B \in \mathcal B$, the map
$
\omega \mapsto \nu_\omega(B)
$
is measurable. 
By standard measure theoretical arguments, we easily deduce that for every $\varphi \in C(X)$,
\[
\omega \mapsto \int_X \varphi\, d\mu_\omega
\]
is measurable. Now, recall that since $X$ is compact separable metric,   the Wasserstein-$1$ distance $W_1$ metrises the weak$^*$ topology; see e.g. \cite[Theorem 6.9]{V09}. Hence the corresponding Borel $\sigma$-algebras coincide.
The weak$^*$  topology on $\mathcal M_1(X)$ is the initial topology generated by the maps
\[
\mu \mapsto \int_X \varphi\, d\mu, \qquad \varphi \in C(X).
\]
Therefore, since all compositions
\[
\omega \mapsto \int_X \varphi\, d\nu_\omega
\]
are measurable and  $\mathcal{M}_1(X)$ is a compact separable metric whose Borel $\sigma$-algebra is generated by these linear functionals, it follows that the map
$
\omega \mapsto \mu_\omega
$
is measurable as a map into $\mathcal M_1(X)$ endowed with the weak topology, and hence also measurable with respect to the Wasserstein-$1$ distance $W_1$.
\end{proof}

\begin{lemma} \label{le.fixedeq}
Let $\mu$ be a $T$-invariant probability measure on $\Omega \times X$,   let
$ \{\nu_\omega\}_{\omega \in \Omega}$ be the factorisation of $\mu$  with respect to $\eta = \pi_* \mu \ll m$ and 
$\rho = \frac{d\eta}{dm}$. Then, for $\eta$-almost every $\omega \in \Omega$,
\[
\nu_{\omega}
=
\sum_{\theta \in f^{-1}(\{\omega\})}
\frac{\rho(\theta)}{\rho(\omega)\, |\det Df(\theta)|}
\,(g_{\theta})_* \nu_{\theta}.
\]
\end{lemma}

\begin{proof}
Let $A \subset \Omega$ and $B \subset X$ be measurable sets. Since $\mu$ is $T$-invariant, we have
\[
\mu(A \times B)
=
\mu\bigl(T^{-1}(A \times B)\bigr).
\]
We compute the right-hand side. By definition of $T$,
\[
T^{-1}(A \times B)
=
\{(\omega,x) \in \Omega \times X :
f(\omega) \in A,\; g_{\omega}(x) \in B \}.
\]
Using the disintegration of $\mu$, we obtain
\[
\mu\bigl(T^{-1}(A \times B)\bigr)
=
\int_\Omega
\nu_{\omega}\bigl(g_{\omega}^{-1}(B)\bigr)
\, \mathbf{1}_A(f(\omega)) \, \rho(\omega)\, dm(\omega).
\]
We now perform a change of variables using the map $f : \Omega \to \Omega$. Since $\eta \ll m$ and $\rho=d\eta/dm$, the Perron--Frobenius formula gives
\[
\int_\Omega \varphi(\omega)\, d\eta(\omega)
=
\int_\Omega \sum_{\theta \in f^{-1}(\{\omega\})}
\frac{\rho(\theta)}{|f'(\theta)|}\,
\varphi(\omega)\, dm(\omega).
\]
Applying this identity to the integrand above yields
\[
\mu(A \times B)
=
\int_A
\sum_{\theta \in f^{-1}(\{\omega\})}
\frac{\rho(\theta)}{|f'(\theta)|}
\, \nu_{\theta}\bigl(g_{\theta}^{-1}(B)\bigr)
\, dm(\omega).
\]
On the other hand,
\[
\mu(A \times B)
=
\int_A \nu_{\omega}(B)\, \rho(\omega)\, dm(\omega).
\]
By uniqueness of disintegration, we deduce that for $m$-a.e. $\omega$,
\[
\nu_{\omega}(B)
=
\sum_{\theta \in f^{-1}(\{\omega\})}
\frac{\rho(\theta)}{\rho(\omega)\, |f'(\theta)|}
\, \nu_{\theta}\bigl(g_{\theta}^{-1}(B)\bigr),
\]
which completes the proof.
\end{proof}

 Let $D_T$ denote the set of discontinuity points of the skew-product map $T$. The proof of the next lemma can be found in \cite[Lemma 2.4]{ZZ11}.

\begin{lemma}\label{le.zz}
  Let $\mu_n, \mu_0 \in \mathcal{M}_1(\Omega\times X)$ be such that $\mu_n \to \mu_0$ in the weak$^*$ topology  and
\begin{equation*}
    \mu_0(D_T) = 0.  
\end{equation*}
Then $T_* \mu_n \to T_* \mu_0$ in in the weak$^*$ topology.
\end{lemma}

Recall that $\eta$ is the $f$-invariant probability measure  given by \ref{as.base2}.

\begin{proposition}
There is a  unique $T$-invariant measure $\mu$ such that $\pi_* \mu = \eta$.
Moreover, 
if $B(\eta)$ is the basin of $\eta$ for the map $f$ and $B(\mu)$ is the basin of   $\mu$ for the skew-product map $T$, then
\[
B(\mu) = B(\eta) \times X.
\]
\end{proposition}

\begin{proof}
First we prove the existence.
By Corollary~\ref{co.fixK}, the operator $\mathcal K$ admits a fixed point
$\{\nu_\omega\}_{\omega\in\Omega}$ in $\mathbf P$. Define the measure
\[
\mu = \int_\Omega \nu_\omega \, d\eta(\omega).
\]
By~\cite[Theorem A]{AB26}, this measure   is $T$-invariant and, clearly, satisfies $\pi_*\mu=\eta$.

Now we prove uniqueness.
Let $\mu$ be a $T$-invariant probability measure with $\pi_*\mu=\eta$. By \cite[Proposition 1.4.3]{A98}, there exists an $\eta$-almost everywhere unique disintegration $\{\nu_\omega\}_{\omega\in\Omega}$ of $\mu$.
By Lemma~\ref{le.uniqueP}, the family $\bm{\nu}=\{\nu_\omega\}_{\omega\in\Omega}$ belongs to $\mathbf P$ and by Lemma~\ref{le.fixedeq} it satisfies $\mathcal K\bm{\nu}=\bm{\nu}$. 
Since Corollary~\ref{co.fixK} gives that the fixed point of $\mathcal K$ in $\mathbf P$ is unique, we are done.

Now we prove the equality involving the two basins. First we show that $B(\mu)\subset  B(\eta)\times X$.

Let $z=(\omega,x)\in \Omega\times X$, and define for $n\ge1$  the empirical measures
\[
\varepsilon_n(z)=\frac{1}{n}\sum_{k=0}^{n-1}\delta_{T^k(z)}\quad\text{and}\quad\bar\varepsilon_n(z)=\frac{1}{n}\sum_{k=0}^{n-1}\delta_{f^k(\omega)}.
\]
Using the semiconjugacy relation $\pi\circ T=f\circ \pi$, where $\pi$ is the canonical projection onto $\Omega$, we obtain
 \begin{equation}\label{eq.empirical}
\pi_*\varepsilon_n(z)
=
\frac{1}{n}\sum_{k=0}^{n-1}\pi_*\delta_{ T^k(\omega,x)}
=
\frac{1}{n}\sum_{k=0}^{n-1}\delta_{\pi(T^k(\omega,x))}
=
\frac{1}{n}\sum_{k=0}^{n-1}\delta_{f^k(\omega)}
=
\bar\varepsilon_n(\omega).
\end{equation}
To prove the forward inclusion $B(\mu)\subset B(\eta)\times X$, assume  $z=(\omega,x)\in B(\mu)$. Then $\varepsilon_n(z)\to\mu$ in the weak$^*$ topology.
Since the push-forward map $\pi_*:  \mathcal M_1(\Omega\times X)\to \mathcal M_1(\Omega)$ is continuous in the weak$^*$ topology, we obtain
\[
\pi_*\varepsilon_n(z)\xrightarrow[n\to\infty]{} \pi_*\mu=\eta.
\]
By \eqref{eq.empirical}, this implies $\varepsilon_n(\omega)\to \eta$, hence $\omega\in B(\eta)$.
Therefore $z\in B(\eta)\times X$.

Finally, we prove the backward inclusion $B(\eta)\times X\subset B(\mu)$.
Let $(\omega,x)\in B(\eta)\times X$. Then
 \begin{equation}\label{eq.empeta}
\bar\varepsilon_n(\omega)\to \eta.
\end{equation}
Since $\mathcal M_1(\Omega\times X)$ is compact in the weak$^*$ topology, the sequence $(\varepsilon_n(z))_{n }$ admits at least one convergent subsequence  $(\varepsilon_{n_k}(z))$: for some $\nu\in \mathcal M_1(\Omega\times X)$, 
 \begin{equation}\label{eq.comver}
\varepsilon_{n_k}(z)\xrightarrow{w^*} \nu.
\end{equation}
Since $T$ is not necessarily continuous, the $T$-invariance of $\nu$ is not immediate. However,
observe that
\[
T_*\varepsilon_n(z)
=
\frac{1}{n}\sum_{k=0}^{n-1}\delta_{T^{k+1}(z)}
=
\varepsilon_n(z)
-\frac{1}{n}\delta_z
+\frac{1}{n}\delta_{T^n(z)}.
\]
Hence, for every $\varphi\in C(\Omega\times X)$,
\[
\int \varphi\,d(T_*\varepsilon_n(z))-\int \varphi\,d\varepsilon_n(z)
=
\frac{1}{n}\big(\varphi(T^n z)-\varphi(z)\big)\longrightarrow 0.
\]
This implies
 \begin{equation}\label{eq.twoseq}
T_*\varepsilon_n(z)-\varepsilon_n(z)\xrightarrow{w^*}0.
\end{equation}
We now pass to the subsequence $n_k$. We claim that
 \begin{equation}\label{eq.claimco}
T_*\varepsilon_{n_k}(z)\xrightarrow{w^*} T_*\nu.
\end{equation}
To justify this, we verify the hypothesis of Lemma~\ref{le.zz}. First, note that by~\eqref{eq.empirical}  and~\eqref{eq.empeta}
\[
\pi_*\varepsilon_{n_k}(z)=\varepsilon_{n_k}(\omega)\xrightarrow{w^*} \eta.
\]
It follows from the continuity of $\pi_*$ and \eqref{eq.comver} that
\[
\pi_*\nu=\eta.
\]
Moreover, since $ D_T=\pi^{-1}(D_f)$ and  $m( D_f)=0$ by assumption, using $\eta\ll m$, we obtain
\[
\nu(D_T)= \nu(\pi^{-1}(D_f))=\pi_*\nu(D_f)=\eta(D_f)=0.
\]
Therefore, Lemma~\ref{le.zz} applies and yields~\eqref{eq.claimco}.
Combining it with \eqref{eq.twoseq}, we conclude that
$
T_*\nu=\nu,
$
so $\nu$ is $T$-invariant.
Moreover, since $\pi_*\nu=\eta$ and $\mu$ is the unique $T$-invariant measure $\mu$ such that $\pi_* \mu = \eta$, we conclude that
\[
\nu=\mu.
\]
Thus every convergent subsequence of $\varepsilon_n(z)$ converges to $\mu$, and therefore the full sequence converges to $\mu$.
This proves $z\in B(\mu)$.
\end{proof}

\section{Decay of Correlations}\label{se.DC}
Here we prove Theorem~\ref{th.DC}. Let $\bm\nu$ be the unique fixed point of
$\mathcal K$ in $\mathbf P$, and let
\[
\mu
=
\int_\Omega \nu_\omega\, d\eta(\omega).
\]

Let  $\Phi,\Psi:\Omega\times X\to \mathbb R$ be Lipschitz continuous functions. 
With no loss of generality we assume
\[
 \Phi\geq 0
\qand
\int \Phi\, d\mu =1 .
\]
Define the probability measures 
$$\mu_\Phi=\Phi\mu\qand  \eta_\Phi 
=
\pi_*\mu_\Phi 
.$$  Set for each $k\ge1$
 \begin{equation}\label{eq.pushm}
\mu_\Phi^k
=
T_*^k(\mu_\Phi)
\qand
\eta_\Phi^k
=
\pi_*\mu_\Phi^k,
\end{equation}
Consider the disintegration of $\mu_\Phi^k$ with marginal $\eta_\Phi^k$ in the base
\begin{equation}\label{eq.samplek}
\mu_\Phi^k
=
\int \nu_{\Phi,\omega}^k\, d\eta_\Phi^k(\omega),
\end{equation}
for some measurable family of sample measures $\{\nu_{\Phi,\omega}^k\}_{\omega\in\Omega}$.
Hence,
\begin{align*}
\mathrm{Corr}_\mu(\Phi,\Psi\circ T^n)
&=
\int_{\Omega\times X} \Phi \cdot( \Psi\circ T^n) d\mu
-
\int_{\Omega\times X} \Psi\, d\mu
\\
&=
\int_{\Omega\times X}
\Psi\circ T^{n-k}\,
d\mu_\Phi^k
-
\int_{\Omega\times X}
\Psi\, d\mu
\\
&=
\int_\Omega
\int_X
\Psi\circ T^{n-k}\,
d\nu_{\Phi,\omega}^k\,
d\eta_\Phi^k(\omega)
-
\int_\Omega
\int_X
\Psi\, d\nu_\omega\, d\eta(\omega)
\\
&=
\underbrace{\int_\Omega
\int_X
\Psi\circ T^{n-k}\,
d\nu_{\Phi,\omega}^k\,
d(\eta_\Phi^k-\eta)(\omega)}_{\text{Base Term}}
\\
&\qquad
+
\underbrace{\left(
\int_\Omega
\int_X
\Psi\circ T^{n-k}\,
d\nu_{\Phi,\omega}^k\,
d\eta(\omega)
-
\int_\Omega
\int_X
\Psi\, d\nu_\omega\, d\eta(\omega)
\right)}_{\text{Fibre Term}}
.
\end{align*}
Next, we estimate the fibre term and the base term separately to yield the conclusion of Theorem~\ref{th.DC}. Fix
$\kappa\in(0,1)$ and choose an integer   
\[
k=k(n)\sim \kappa n.
\]
Note that as $d$ is the distance induced by the product Riemannian metric on $\Omega\times X$, it follows that for all $\omega,\omega'\in\Omega$ and $x,x'\in X$,
\begin{equation*}\label{eq.distance}
d\bigl((\omega,x),(\omega',x)\bigr)
=
d_\Omega(\omega,\omega')
\qand
d\bigl((\omega,x),(\omega,x')\bigr)
=
d_X(x,x').
\end{equation*}
In particular, if $\Phi:\Omega\times X\to\mathbb R$ is Lipschitz, then for every $\omega\in\Omega$ and $x\in X$, the sections
\[
\Phi(\omega,\cdot):X\to\mathbb R
\qquad\text{and}\qquad
\Phi(\cdot,x):\Omega\to\mathbb R
\]
are Lipschitz, and their Lipschitz constants are bounded above by the Lipschitz constant of~$\Phi$.
\subsection{Fibre Term} 
Consider the family of sample measures for $\mu_\Phi^k$ given in~\eqref{eq.samplek} 
\[
\bm\nu_\Phi^k
=
\left\{\nu_{\Phi,\omega}^k\right\}_{\omega\in\Omega}.
\]
By the duality in~\cite[Lemma~2.4]{AB26},  
\begin{align*}
\int_\Omega
\int_X
\Psi\circ T^{n-k}\,
d\nu_{\Phi,\omega}^k\,
d\eta(\omega)
=
\int_\Omega
\int_X
\Psi\,
d\!\left(
\mathcal K^{\,n-k}\bm\nu_{\Phi}^k
\right)_\omega
\, d\eta(\omega),
\end{align*}
and so the fibre term becomes
\begin{align}
\bigg|\int_\Omega
\int_X
\Psi\,
d\!\left(
\mathcal K^{\,n-k}\bm\nu_{\Phi}^k
\right)_\omega
\, d\eta(\omega)
&-
\int_\Omega
\int_X
\Psi\, d\nu_\omega\, d\eta(\omega)
\bigg|\nonumber\\
&
\le  
\int_\Omega \left|
\int_X
{\Psi}\,
d(\mathcal K^{\,n-k}\bm\nu_{\Phi}^k)_\omega \,
-
\int_X
{\Psi} \,
d\nu_\omega \right|
d\eta(\omega)\nonumber\\
&
\le
\lip(\Psi) \int_\Omega
W_1\Big(
(\mathcal K^{\,n-k}\bm\nu_{\Phi}^k)_\omega,
\nu_\omega
\Big)d\eta(\omega).
\label{fiterm}
\end{align}
Since $\bm\nu$ is the fixed point of $\mathcal K$, by Proposition~\ref{pr.Kcontraction}
we have for $\eta$-almost every $\omega\in\Omega$
\[
W_1\Big(
(\mathcal K^{\,n-k}\bm\nu_{\Phi}^k)_\omega,
\nu_\omega
\Big) 
\le 
d_{\mathbf P}(\mathcal K^{\,n-k}\bm\nu_{\Phi}^k,\mathcal K^{\,n-k}\bm\nu)
\le
r^{n-k}
d_{\mathbf P}(\bm\nu_{\Phi}^k,\bm\nu)^q \le r^{n-k} \diam(X)^q .
\]

Therefore, by \eqref{fiterm}, 
\[
|\text{Fibre Term}|\le \diam(X)^q
r^{n-k}\lip(\Psi)
=
\diam(X)^{q}
r^{(1-\kappa)n}\lip(\Psi)  ,
\]
which motivates the choice
$
s=r^{1-\kappa}.
$
 
  \subsection{Base Term}
  To study the base term we first show that  $\eta_\Phi\ll \eta$ and that the density of  $\eta_\Phi$ with respect to $\eta$ is Hölder continuous.
\begin{lemma}\label{le:marginal}
The probability measure   $\eta_\Phi$ is absolutely continuous with respect to $\eta$, and its density is given by 
\[
h_\Phi(\omega)
=
\int_X \Phi(\omega,x)\, d\nu_\omega(x),
\]
Moreover,   $h_\Phi$ is $q$-Hölder continuous and there exists   $C>0$ (independent of $\Phi$) such that
\[
[h_\Phi]_q
\le
C \Lip (\Phi) .
\]
\end{lemma}

\begin{proof}
Let $A\subset\Omega$ be a Borel measurable set. Since $\eta_\Phi=\pi_*\mu_\Phi$,
\[
\eta_\Phi(A)
=
\mu_\Phi(A\times X).
\]
Using the definition of $\mu_\Phi$ and disintegrating $\mu$ along the fibres gives
\[
\eta_\Phi(A)
=
\int_{A\times X}\Phi(\omega,x)\, d\mu(\omega,x)
=
\int_A
\left(
\int_X \Phi(\omega,x)\, d\nu_\omega(x)
\right)
d\eta(\omega).
\]
This shows that 
$\eta_\Phi\ll\eta$ and
\[
\frac{d\eta_\Phi}{d\eta}=h_\Phi.
\]
%
We now prove the $q$-Hölder regularity of $h_\Phi$. Let $\omega,\omega'\in\Omega$. Then
\begin{align*}
|h_\Phi(\omega)-h_\Phi(\omega')|
&=
\left|\int_X \Phi(\omega,x)\, d\nu_\omega(x)
-
\int_X \Phi(\omega',x)\, d\nu_{\omega'}(x)\right|\\
&\le
\left|
\int_X
\bigl(
\Phi(\omega,x)-\Phi(\omega',x)
\bigr)
\, d\nu_\omega(x)
\right|
 +
\left|
\int_X \Phi(\omega',x)\,
d(\nu_\omega-\nu_{\omega'})(x)
\right|.
\end{align*}
For the first term, 
\[
\left|
\int_X
\bigl(
\Phi(\omega,x)-\Phi(\omega',x)
\bigr)
\, d\nu_\omega(x)
\right|
\le
 \Lip (\Phi)\,
d_\Omega(\omega,\omega').
\]
For the second term, Kantorovich duality 
yields
\[
\left|
\int_X \Phi(\omega',x)\,
d(\nu_\omega-\nu_{\omega'})(x)
\right|
\le
 \Lip (\Phi)\,
W_1(\nu_\omega,\nu_{\omega'}).
\]
By Corollary~\ref{th.holder}, there exists some constant $H>0$ such that
\[
W_1(\nu_\omega,\nu_{\omega'})
\le
H  d_\Omega(\omega,\omega')^q.
\]
Therefore
\[
|h_\Phi(\omega)-h_\Phi(\omega')|
\le
\Lip (\Phi) d_\Omega(\omega,\omega')
+
  H  \Lip(\Phi)
d_\Omega(\omega,\omega')^q,
\]
which clearly gives  that $h_\Phi$ is a $q$-Hölder continuous function  with  its Hölder seminorm bounded by $C \Lip (\Phi)$, for some uniform constant $C>0$.
\end{proof}

By~\eqref{eq.pushm}, we have
 \begin{equation}\label{eq.etak}
\eta_\Phi^k
=
\pi_*\mu_\Phi^k
=
\pi_* T_*^k\mu_\Phi 
=
f_*^k \pi_*\mu_\Phi 
=
f_*^k \eta_\Phi.
\end{equation}
and by Lemma~\ref{le:marginal}
$$
\frac{d\eta_\Phi}{d\eta}(\omega)=h_\Phi(\omega)=\int_X \Phi(\omega,x)\, d\nu_\omega(x).
$$
Recall that the \emph{total variation} of  any signed finite  Borel  measure $\zeta$ on $\Omega$   is defined as
\[
\|\zeta\|_{\mathrm{TV}}
=
\sup_{\|\psi\|_\infty\le1}
\left|
\int  \psi\, d\zeta
\right|.
\]

\begin{lemma} 
\label{lem:tv-convergence}
Let $r$ be the rate function in~\ref{as.dc}. Then, for all $n\ge1$
\[
\left\|(f^n)_*\eta_\Phi-\eta\right\|_{\mathrm{TV}}
\le
r(n)\,
\left\|h_\Phi -1\right\|_{q}.
\]
\end{lemma}

\begin{proof}
For every bounded measurable observable
$
\psi:\Omega\to\mathbb R,
$
we have
\[
\int_\Omega \psi\, d(f^n)_*\eta_\Phi
=
\int_\Omega \psi\circ f^n\, d\eta_\Phi
=
\int_\Omega (\psi\circ f^n)\, h\, d\eta.
\]
Since $\eta$ is $f$-invariant,
\[
\int_\Omega \psi\, d\eta
=
\int_\Omega \psi\circ f^n\, d\eta.
\]
Therefore
\[
\int_\Omega
\psi\,
d\bigl((f^n)_*\eta_\Phi-\eta\bigr)
=
\int_\Omega
(\psi\circ f^n)(h_\Phi-1)\,
d\eta,
\]
where $h_\Phi$ is the density of $\eta_\Phi$ with respect to $\eta$, given by Lemma~\ref{le:marginal}.
Since
\[
\int_\Omega (h_\Phi-1)\, d\eta
=
\int_\Omega h_\Phi\, d\eta-1
=
0,
\]
the decay of correlations assumption~\ref{as.dc} gives
\[
\left|
\int_\Omega
(\psi\circ f^n)(h_\Phi-1)\,
d\eta
\right|
\le
r(n)\,
\|\psi\|_\infty\,
\|h_\Phi-1\|_{q}.
\]
Taking the supremum over all observables satisfying
$
\|\psi\|_\infty\le1
$
yields
\[
\left\|(f^n)_*\eta_\Phi-\eta\right\|_{\mathrm{TV}}
\le
r(n)\,
\|h_\Phi-1\|_{q},
\]
thus completing the proof.
\end{proof}

Now we are able to obtain the estimate for the base term.  By definition of total variation, 
\begin{align*}
\left|
\int_\Omega
\left(
\int_X
\Psi\circ T^{n-k}(\omega,x)\,
d\nu_{\Phi,\omega}^k(x)
\right)
\, d(\eta_\Phi^k-\eta)(\omega)
\right|
&\le
\|\Psi\|_\infty
\left\|
\eta_\Phi^k-\eta
\right\|_{\mathrm{TV}}
.
\end{align*}
Applying~\eqref{eq.etak} and Lemma~\ref{lem:tv-convergence}, we obtain
\[
\left\|
\eta_\Phi^k-\eta
\right\|_{\mathrm{TV}}
=\left\|
(f_*^k\eta_\Phi)-\eta
\right\|_{\mathrm{TV}}
\le
 r(k)\,
\left\|
h_\Phi-1
\right\|_{q}.
\]
By Lemma~\ref{le:marginal}, there exists some $C\ge 1$ such that
\[
\left[
h_\Phi -1
\right]_{q}
=
\left[
h_\Phi  
\right]_{q}
\le
C\lip(\Phi),
\]
so that
\[
\left|
\text{Base Term}
\right|
\le
 r(k)\,
\|\Psi\|_\infty\,
\left\|
h_\Phi-1
\right\|_{q}
=
C r(\kappa n)\,
\|\Psi\|_\infty\,
\|\Phi\|_{\lip}.
\]
We have thus established Theorem~\ref{th.DC}.

\part{Applications}\label{part.applications}

In this part, we apply the abstract results   to several classes of skew-product systems with nontrivial fibre dynamics. We begin with heterochaotic baker maps, which serve as prototypical examples exhibiting coexistence of invariant sets with different unstable dimensions and allow us to extend previous results beyond conservative regimes to parameter regions where invariant measures are not explicitly known, establishing existence of physical measures and exponential decay of correlations under average fibre contraction. We then consider a smooth skew-product family acting as a nonuniformly contracting analogue of the classical Viana maps, where fibre contraction dominates despite intermittent expansion, showing that our methods apply beyond the standard nonuniformly expanding setting. Finally, we study solenoidal attractors with a Pomeau--Manneville type intermittent base, where the neutral fixed point leads to only polynomial decay of correlations, yet the system still fits within our framework due to average fibre contraction. Together, these examples illustrate the robustness of our theory across systems of different regularity and mixing behaviour, yielding existence and uniqueness of physical measures and quantitative rates of decay of correlations.

\section{Heterochaos Baker Maps}

We apply the two main theorems of Part~\ref{part.general}  to an extended class of heterochaotic baker maps introduced in~\cite{STY21}, including parameter regimes where Lebesgue measure is not invariant. We consider both a one-parameter family of non-invertible maps on the unit square and a two-parameter family of invertible maps on the unit cube. In contrast to~\cite{STY21, T25, TT25}, which treated only conservative regimes, we obtain existence of physical measures and exponential decay of correlations throughout the parameter region where the fibre dynamics is contracting on average.

\subsection{Non-Invertible Version}
We consider $\Omega=X=[0,1]$ so that $\Omega\times X$ is the unit square \([0,1]^2\).
We fix an integer \(M\ge 2\)  and take 
\[
a\in \left(0,\frac1M\right).
\]
Define the piecewise linear base map
\(
f_a:[0,1]\to[0,1]
\)
by
\[
f_a(\omega)=
\begin{cases}
\dfrac{\omega-(i-1)a}{a},
&
\omega\in[(i-1)a,ia),
\qquad i=1,\ldots,M,
\\[2ex]
\dfrac{\omega-Ma}{1-Ma},
&
\omega\in[Ma,1].
\end{cases}
\]
Given $b\in(0,1)$ and $\omega\in[0,1]$, define a piecewise linear   fibre map
\(
g_{a,b}:[0,1]^2\to[0,1]
\)
by
\[
g_{a,b}(\omega,x)=
\begin{cases}
b x+\dfrac{(i-1)(1-b)}{M},
&
(\omega,x)\in[(i-1)a,ia)\times [0,1],
\qquad i=1,\ldots,M,
\\[2ex]
2\mathop{\min}\limits_{n\in\mathbb Z}
\left|
\dfrac{Mx}{2}-n
\right|,
&
(\omega,x)\in[Ma,1]\times[0,1].
\end{cases}
\]
For each $\omega\in[Ma,1]$, the fibre map  \(g_{a,b}(\omega,\cdot):[0,1]\to [0,1] \)  is continuous and piecewise affine, with slopes~\(\pm M\), therefore Lipschitz continuous with $\lip(g_{a,b}(\omega,\cdot))=M$. 
The \emph{two-dimensional heterochaos baker map}
is the skew-product map \(
T_{a,b}:[0,1]^2\to[0,1]^2
\)
 defined for $(\omega,x)\in [0,1]^2$ by
\[
T_{a,b}(\omega,x)=
\left(
f_a(\omega),
g_{a,b}(\omega,x)
\right).
\]
\subsubsection{Conservative Case}
For the particular choice \(b=\frac1M\), the family \(T_{a,\frac1M}\) may be viewed as a continuous variant of the model introduced in~\cite{STY21}.\footnote{In~\cite{STY21}, the fibre map on the expanding branch corresponding to $\omega\in[Ma,1]$ is \(x\mapsto Mx \pmod 1\). Here we replace it by the continuous \emph{saw-map}
\[
x\mapsto 2\min_{n\in\mathbb Z}\left|\frac{Mx}{2}-n\right|,
\]
solely in order to obtain Lipschitz continuous fibre maps. This change does not affect the qualitative ergodic and statistical properties of the system, since  the map \(x \mapsto Mx \bmod 1\) and the saw map are measurably conjugate.}
These maps preserve Lebesgue measure on \([0,1]^2\), and it was proved in~\cite[Theorem 1.2]{STY21} that Lebesgue measure is ergodic. Consequently, it is the unique physical measure and its basin has full Lebesgue measure.
The maps \(T_{a,1/M}\) were classified in~\cite[Section 1.3]{T25} according to the sign of the Lyapunov exponent in the fibre (vertical) direction:
\begin{itemize}
\item \(a\in \left(0,\frac{1}{2M}\right)\): \emph{mostly expanding centre};
\item \(a\in \left(\frac{1}{2M},\frac{1}{M}\right)\): \emph{mostly contracting centre};
\item \(a=\frac{1}{2M}\): \emph{mostly neutral centre}.
\end{itemize}
In this conservative case, \(a=\frac{1}{2M}\) is therefore the bifurcation parameter at which the central Lyapunov exponent vanishes and changes sign. It was shown in~\cite[Theorem B]{T25} that \(T_{a,\frac1M}\) exhibits exponential decay of correlations for all
\(
a\in \left(0,\tfrac{1}{M}\right)\setminus\left\{\tfrac{1}{2M}\right\},
\)
whereas~\cite[Main Theorem]{T25} established that the critical map \(T_{\frac1{2M},\frac1M}\) has polynomial decay of correlations of order \(n^{-3/2}\); the results in~\cite{T25} also give this rate is  optimal.

    \subsubsection{General Case}

For general values of \(b\), Lebesgue measure is no longer necessarily invariant. A natural question is therefore whether the maps \(T_{a,b}\) admit a physical measure, whether this measure is unique, and what statistical properties it enjoys. We now show that the results developed in Part~\ref{part.general} apply to these maps throughout a large region of parameter space. We will show that assumptions~\ref{as.fibre}--\ref{as.dc} hold within a certain range of parameter values, the latter with an exponential rate function. Therefore, the following result follows as a consequence of Theorems~\ref{th.mainA} and~\ref{th.DC}.

\begin{maincorollary}\label{co.SRBDC}
Assume that 
$a\in (0,\tfrac1M)$ and
$ b
\in \big(0,
M^{1-\frac{1}{Ma}}\big)
$. Then,
\begin{enumerate}
\item $T_{a,b}$ has a unique physical measure $\mu_{a,b}$ whose basin covers Lebesgue almost all of $[0,1]^2$;
\item  there exist $C > 0$ and ${s\in(0,1)}$ such that for all  Lipschitz    $\Phi, \Psi: [0,1]^2\to\mathbb R$ and $n\ge1$ 
\[
\left| \int \Phi \cdot \Psi \circ T_{a,b}^n \, d\mu_{a,b} - \int \Phi \, d\mu_{a,b} \int \Psi \, d\mu_{a,b} \right| \le 
C\left( \|\Phi\|_{\lip} + \|\Psi\|_{\lip} \right) s^n .
\]
\end{enumerate}

\end{maincorollary}

It is worth noting that, unlike the conservative family studied in~\cite{T25}, the above result does not immediately extend to the mostly expanding regime. In the case $b=\frac1M$, Takahashi exploited a special symmetry of the model which relates the dynamics for parameters $a$ and $\frac1M-a$, thereby reducing one regime to the other. This argument relies crucially on the exact matching between the expanding and contracting fibre slopes, which is tied to the invariance of Lebesgue measure. For the present two-parameter family, no analogous symmetry structure is apparent. 

We now determine the  parameters values  for which  assumptions~\ref{as.fibre}--\ref{as.dc} hold.
First note that~\ref{as.fibre}--\ref{as.base} are easily verified for all admissible parameters. Indeed, the base map \(f_a\) is a piecewise expanding full-branch map with \(N=M+1\) smoothness domains. A straightforward computation of the weights defined in~\eqref{eq.weight} shows that they are constant on each smoothness domain and given by
\[
p_1=\cdots=p_M=a,
\qquad
p_{M+1}=1-Ma.
\]
Moreover, the fibre maps \(g_{a,b}(\omega,\cdot):[0,1]\to [0,1]\) have Lipschitz constants that are constant on each smoothness domain of the base map. Denoting these constants by \(L_1,\ldots,L_M,L_{M+1}\), we obtain
\[
L_1=\cdots=L_M=b,
\qquad
L_{M+1}=M.
\]
Therefore, assumption~\ref{as.mean} is equivalent to
\[
\sum_{i=1}^{M+1} p_i \log L_i <0,
\]
which in the present setting becomes
\[
Ma\log b +(1-Ma)\log M <0.
\]
Solving for \(b\), we obtain
\[
\log b
<
-\frac{1-Ma}{Ma}\log M,
\]
and hence
 \begin{equation}\label{eq.alpha}
b
<
M^{1-\frac{1}{Ma}}.
\end{equation}
Since \(Ma<1\), we have \(1-\frac{1}{Ma}<0\). It follows that
  this upper bound for \(b\) is always strictly smaller than \(1\). 
  Finally, assumption~\ref{as.holder} is trivially satisfied, while assumption~\ref{as.dc} is known to hold for the base map \(f_a\) with an exponential rate function \(r\).

  \subsection{Invertible Version}

A three-dimensional version is obtained by adjoining a stable direction to the non-invertible model, resulting in a map whose failure of invertibility is confined to the boundaries of the smoothness domains. 
We take \(\Omega=[0,1]\) and \(X=[0,1]^2\), so that \(\Omega\times X=[0,1]^3\) is the unit cube. Points in \(\Omega\times X\) are written as \((\omega,x,y)\), where \(\omega\), \(x\), and \(y\) correspond to the unstable, central, and stable directions, respectively.
We fix an integer $M\ge 2$ and take
\[
a, c\in\left(0,\frac1M\right).
\]
Define
\(
h_c:[0,1]^3\to[0,1]
\)
by
\[
h_c(\omega,x,y)=
\begin{cases}
(1-Mc)y,
&
(\omega,x)\in[(i-1)a,ia)\times[0,1],
\qquad i=1,\ldots,M,
\\[2ex]
cy+\dfrac{i-1}{M},
&
(\omega,x)\in
[Ma,1]\times
\left[\dfrac{i-1}{M},\dfrac{i}{M}\right),
\qquad i=1,\ldots,M.
\end{cases}
\]
The \emph{three-dimensional heterochaos baker map} is the map
\(
\widehat T_{a,b,c}:[0,1]^3\to[0,1]^3
\)
defined by
\[
\widehat T_{a,b,c}(\omega,x,y)
=
\Bigl(
f_a(\omega),
g_{a,b}(\omega,x),
h_c(\omega,x,y)
\Bigr),
\]
where \(f_a\) and \(g_{a,b}\) are the maps introduced in the previous subsection. 
The map \(\widehat T_{a,b,c}\) is a partially hyperbolic extension of \(T_{a,b}\). The unstable coordinate \(\omega\) is uniformly expanded by the base map \(f_a\), the stable coordinate \(y\) is uniformly contracted, while the central coordinate \(x\) alternates between contraction and expansion according to the itinerary of the orbit.

\subsubsection{Conservative Case}

For the particular choice \(b=\frac1M\) and
\[
a+c=\frac1M,
\]
the family \(\widehat T_{a,\frac1M,\frac1M-a}\) may be viewed as the invertible heterochaos baker maps introduced in~\cite{STY21}. As in the two-dimensional case, the modification consists only in replacing the map \(x\mapsto Mx \pmod 1\) by its continuous saw-map counterpart in the central direction.
For these parameter values, Lebesgue measure is invariant.
As before, the statistical properties of the system depend on the sign of the central Lyapunov exponent,
and \(a=c=\frac{1}{2M}\) is therefore the bifurcation parameter at which the central Lyapunov exponent vanishes and changes sign. It was also shown in~\cite[Theorem B]{T25} that \(\widehat T_{a,\frac1M,\frac1M-a}\) exhibits exponential decay of correlations for all
\(
a\in \left(0,\tfrac{1}{M}\right)\setminus\left\{\tfrac{1}{2M}\right\},
\)
whereas~\cite[Main Theorem]{T25} established that the critical map \(\widehat T_{\frac1{2M},\frac1M,\frac1{2M}}\) has polynomial decay of correlations of order \(n^{-3/2}\), with this rate being optimal.

\subsubsection{General Case}
As in the non-invertible case, 
for general values of \(a,b,c\), Lebesgue measure is not necessarily invariant.  We  will show that the results developed in Part~\ref{part.general} also apply to these maps throughout a large region of parameter space. We will show that assumptions~\ref{as.fibre}--\ref{as.dc} hold within a certain range of parameter values, the latter with an exponential rate function. As a consequence of Theorems~\ref{th.mainA} and~\ref{th.DC}, the following holds.

\begin{maincorollary}
Assume that 
\[
a\in\left(0,\frac1M\right),
\qquad
b\in\left(0,
M^{\,1-\frac1{Ma}}\right),
\qquad
c\in\left(\frac{1-M^{\,1-\frac1{Ma}}}{M},
M^{\,1-\frac1{Ma}}\right).
\]
Then \(\widehat T_{a,b,c}\) admits a unique physical measure \(\mu_{a,b,c}\) whose basin has full Lebesgue measure. Moreover, there exist \(C>0\) and \(s\in(0,1)\) such that for all Lipschitz   \(\Phi,\Psi:[0,1]^3\to\mathbb R\) and all \(n\ge1\),
\[
\left|
\int \Phi\cdot \Psi\circ \widehat T_{a,b,c}^{\,n}\,d\mu_{a,b,c}
-
\int \Phi\,d\mu_{a,b,c}
\int \Psi\,d\mu_{a,b,c}
\right|
\le
Cs^n\bigl(\|\Phi\|_{\lip}+\|\Psi\|_{\lip}\bigr).
\]
\end{maincorollary}

We now determine the  parameters values  for which  assumptions~\ref{as.fibre}--\ref{as.dc} hold.
As in the non-invertible case, assumptions~\ref{as.fibre}--\ref{as.base} are readily verified. The base map is again~\(f_a\), and therefore the weights associated to its smoothness domains are
\[
p_1=\cdots=p_M=a,
\qquad
p_{M+1}=1-Ma.
\]
The fibre space is now
\(
X=[0,1]^2,
\)
with coordinates \((x,y)\). Moreover, the fibre maps have Lipschitz constants that are constant on each smoothness domain of the base map. Denoting these constants by \(L_1,\ldots,L_M,L_{M+1}\), we have
\[
L_1=\cdots=L_M=\max\{b,c,1-Mc\},
\qquad
L_{M+1}=M.
\]
Consequently, assumption~\ref{as.mean} becomes
\[
Ma\log\bigl(\max\{b,c,1-Mc\}\bigr)
+
(1-Ma)\log M
<
0.
\]
Equivalently,
\[
\max\{b,c,1-Mc\}
<
M^{\,1-\frac1{Ma}}.
\]
Since \(Ma<1\), the exponent \(1-\frac1{Ma}\) is negative, and therefore
\(
M^{\,1-\frac1{Ma}}<1.
\)
Hence the previous condition is equivalent to the pair of inequalities
\[
b
<
M^{\,1-\frac1{Ma}}
,\qquad
\frac{1-M^{\,1-\frac1{Ma}}}{M}<c
<
M^{\,1-\frac1{Ma}}
.
\]
As in the previous section, assumption~\ref{as.holder} is trivially satisfied, while assumption~\ref{as.dc} is known to hold for the base map \(f_a\) with an exponential rate function \(r\).   

\section{Contracting Viana  Maps}

Here we introduce a skew-product family with \(\Omega = \mathbb{S}^1\) and \(X = [0,1]\). This can be viewed as a variation of Viana maps introduced in~\cite{V97}, featuring nonuniformly contracting fibre maps, in contrast with the classical setting where the fibre dynamics are nonuniformly expanding.
Fixing  $M\ge 2$, consider for $\alpha>0$ the   map
\[
T_\alpha:\mathbb S^1\times [0,1]\longrightarrow \mathbb S^1\times [0,1]
\]
defined by
\[
T_\alpha(\omega,x)
=
\bigl(f(\omega),g_\alpha(\omega,x)\bigr),
\]
where the base dynamics is given by
\[
f(\omega)=M\omega \pmod 1,
\]
and the fibre maps are quadratic maps of the form
\[
g_\alpha(\omega,x)=c_\alpha(\omega)x(1-x),
\]
with 
$
c_\alpha:\mathbb S^1\rightarrow (0,4]
$
a smooth function given by
\[
c_\alpha(\omega)
=
4\exp\bigl(-\alpha+\alpha\sin(2\pi\omega)\bigr).
\]
Note that, as \(\alpha\) varies over \((0,\infty)\), the range of possible values of \(c_\alpha(\omega)\) covers the whole interval~\((0,4]\), and
$
c_\alpha(\omega)=4
$
whenever
$
\sin(2\pi\omega)=1.
$
Therefore, \(c_\alpha\) is a smooth function taking values in \((0,4]\) and attaining its maximum value \(4\) at the unique point \(\omega=1/4\) of \(\mathbb S^1\).
  At this    base point, the fibre map coincides with the logistic map
$
x \mapsto 4x(1-x),
$
which is topologically conjugate to the tent map on \([0,1]\) and smoothly conjugate on \((0,1)\). This clearly ensures that \(T_\alpha\) has no uniformly contracting fibres.
The following result is a consequence of Theorems~\ref{th.mainA} and~\ref{th.DC}.

\begin{maincorollary}\label{co.viana}
Assume that   $\alpha\in ( \log 4,\infty)$. Then,
\begin{enumerate}
\item $T_{\alpha}$ has a unique physical measure $\mu_{\alpha}$ whose basin covers Lebesgue almost all of $\mathbb S^1\times [0,1]$;
\item  there exist $C > 0$ and ${s\in(0,1)}$ such that for all  Lipschitz    $\Phi, \Psi: [0,1]^2\to\mathbb R$ and $n\ge1$ 
\[
\left| \int \Phi \cdot \Psi \circ T_{\alpha}^n \, d\mu_{\alpha} - \int \Phi \, d\mu_{\alpha} \int \Psi \, d\mu_{\alpha} \right| \le 
Cs^n\! \left( \|\Phi\|_{\lip} + \|\Psi\|_{\lip} \right)  .
\]
\end{enumerate}

\end{maincorollary}

To prove Corollary~\ref{co.viana}, we show that for \(\alpha>\log 4\), the map \(T_\alpha\) satisfies assumptions~\ref{as.fibre}--\ref{as.dc}.

First, note that   \ref{as.fibre}--\ref{as.base} are easily verified for all \(\alpha>0\). Indeed, the base map~\(f\) is a piecewise expanding, full-branch map with \(M\) smoothness domains preserving Lebesgue measure~$m$ on $\mathbb S^1$. A straightforward computation of the weights defined in~\eqref{eq.weight} shows that they are constant on each smoothness domain and are given by
\[
p_1 = \cdots = p_M = \frac{1}{M}.
\]We now verify  condition~\ref{as.mean}. Since
$
f(\omega)=M\omega
$
(mod 1),
the set of preimages of any \(\omega\in\mathbb S^1\) is given by
\[
f^{-1}(\omega)
=
\left\{
\frac{\omega+k}{M}
:
k=0,\ldots,M-1
\right\}.
\]
Therefore,~\ref{as.mean} holds iff there exists some $\lambda<0$ such that for $m$-almost every $\omega\in\mathbb S^1 $
\[
\sum_{k=0}^{M-1}
\log c_\alpha\!\left(\frac{\omega+k}{M}\right)\le\lambda.
\]
Using the definition of \(c_\alpha\), we obtain
\[
\log c_\alpha(\omega)
=
\log 4-\alpha+\alpha\sin(2\pi\omega),
\]
and hence
\[
\sum_{k=0}^{M-1}
\log c_\alpha\!\left(\frac{\omega+k}{M}\right)
=
M(\log4-\alpha)
+
\alpha\sum_{k=0}^{M-1}
\sin\!\left(
\frac{2\pi(\omega+k)}{M}
\right).
\]
To evaluate the second term, observe that
\[
\sum_{k=0}^{M-1}
e^{2\pi i(\omega+k)/M}
=
e^{2\pi i\omega/M}
\sum_{k=0}^{M-1}
e^{2\pi i k/M}
=
0,
\]
since the \(M\)-th roots of unity sum to zero. Taking imaginary parts yields
\[
\sum_{k=0}^{M-1}
\sin\!\left(
\frac{2\pi(\omega+k)}{M}
\right)
=
0.
\]
Consequently,
\[
\sum_{k=0}^{M-1}
\log c_\alpha\!\left(\frac{\omega+k}{M}\right)=
M(\log4-\alpha)
\]
for every \(\omega\in\mathbb S^1\). Therefore, taking $\alpha>\log 4$ we ensure \ref{as.mean} for every \(\omega\in\mathbb S^1\). 
Finally, the two conditions in~\ref{as.holder} are readily verified in this setting, and~\ref{as.dc} follows from the standard exponential decay of correlations for the expanding map \(f\).

\section{Polynomially Mixing Dynamics}

We introduce a skew-product family with base space $\Omega=\mathbb{S}^1$ and fibre space
$X=\mathbb D$ as the unit disk in $\mathbb R^2$, featuring  a nonuniformly expanding base dynamics with polynomial decay of correlations. Given parameters $\alpha\in(1,\infty)$ and $\beta\in (0,1/2)$, consider the map
\[
T_{\alpha,\beta}:\mathbb{S}^1\times\mathbb D\longrightarrow \mathbb{S}^1\times \mathbb D
\]
defined by
\[
T_{\alpha,\beta}(\omega,x)
=
\bigl(f_\alpha(\omega),g_\beta(\omega,x)\bigr),
\]
where the base map $f_\alpha:\mathbb{S}^1\to\mathbb{S}^1$ is given by the
one-parameter family of continuous maps introduced in~\cite{GH85} whose decay of correlations  has been studied in~\cite{CHM10}. Here
$\mathbb{S}^1=[-1,1]/\sim$, and $f_\alpha$ is implicitly defined by
\begin{equation}
\omega=
\begin{cases}
\tfrac{1}{2\alpha}\bigl(1+f_\alpha(\omega)\bigr)^\alpha,
& \text{if } 0\leq \omega\leq \tfrac{1}{2\alpha},\\[8pt]
f_\alpha(\omega)+\tfrac{1}{2\alpha}\bigl(1-f_\alpha(\omega)\bigr)^\alpha,
& \text{if } \tfrac{1}{2\alpha}\leq \omega\leq 1,
\end{cases}
\tag{3}
\end{equation}
together with the symmetry condition
\[
f_\alpha(-\omega)=-f_\alpha(\omega),\qquad \omega<0.
\]
For every $\alpha>1$, the map $f_\alpha$ is of class $C^1$ on
$\mathbb{S}^1\setminus\{0\}$ and of class $C^2$ on
$\mathbb{S}^1\setminus(\{0,1\})$. 
The point~$1$ is a neutral fixed
point, with derivative equal to $1$, whereas the derivative becomes unbounded at~$0$. Moreover, $f_\alpha$ preserves the Lebesgue measure; see~\cite{CHM10} for details.
The fibre maps are given for $(\omega,x)\in \mathbb S^1\times\mathbb D$ by
\[
g_\beta(\omega,x)
=
\Bigl(\tfrac12\cos(2\pi\omega),\,\tfrac12\sin(2\pi\omega)\Bigr)
+\beta x.
\]
Since $\beta\in(0,1/2)$, we have that
$g_\beta(\omega,x)\in\mathbb{D}$, for every $(\omega,x)\in \mathbb S^1\times\mathbb D$. 
The following result is a consequence of Theorems~\ref{th.mainA} and~\ref{th.DC}.

\begin{maincorollary}\label{co.inter}
Assume that   $\beta\in ( \log 4,\infty)$. Then,
\begin{enumerate}
\item $T_{\alpha,\beta}$ has a unique physical measure $\mu_{\alpha,\beta}$ whose basin covers Lebesgue almost all of $\mathbb S^1\times [0,1]$;
\item  there exist $C > 0$ and ${s\in(0,1)}$ such that for all  Lipschitz    $\Phi, \Psi: [0,1]^2\to\mathbb R$ and $n\ge1$ 
\[
\left| \int \Phi \cdot \Psi \circ T_{\alpha,\beta}^n \, d\mu_{\alpha,\beta} - \int \Phi \, d\mu_{\alpha,\beta} \int \Psi \, d\mu_{\alpha,\beta} \right| \le 
\frac{C}{n^\gamma}\left( \|\Phi\|_{\lip} + \|\Psi\|_{\lip} \right), \qquad \gamma=\frac1{\alpha-1}.
\]
\end{enumerate}

\end{maincorollary}

To prove Corollary~\ref{co.inter}, we show that for \(\beta>\log 4\), the map \(T_\beta\) satisfies assumptions~\ref{as.fibre}--\ref{as.dc}.

First, note that   \ref{as.fibre}--\ref{as.base} are easily verified for all \(\beta>0\). Indeed, the base map~\(f_\alpha\) is a piecewise full-branch map with two smoothness domains preserving Lebesgue measure~\(m\) on~\(\mathbb S^1\), being  expanding everywhere except at a neutral fixed point.
Furthermore, since the fibre maps $g_\beta(\omega,\cdot) $ contract uniformly by the  factor $\beta$, then \ref{as.mean} is trivially satisfied for any $\lambda\in(\log\beta,1)$.
Finally, the two conditions in~\ref{as.holder} are readily verified in this setting, and the decay of correlation in~\ref{as.dc} holds with a rate function 
$$
r(n)\sim \frac{1}{n^\gamma}, \qquad \gamma=\frac1{\alpha-1},
$$
by~\cite[Proposition 5]{CHM10}.

 \addtocontents{toc}{\medskip}
 
\appendix
 \section{A Fixed Point Lemma}

Although the proof of the following lemma follows from standard arguments, we include it for completeness, as we are not aware of an explicit reference in the literature.

\begin{lemma}\label{le.fix}
Let $(X,d)$ be a complete metric space and let $T:X\to X$ be a mapping. 
Assume that there exist constants $C>0$ and $ q,\tau\in(0,1)$ such that
\[
d(T^n x, T^n y)\le C\, \tau^n\, d(x,y)^q,
\qquad \forall x,y\in X,\ \forall n\in\mathbb{N}.
\]
Then $T$ has a unique fixed point $x^*\in X$, and for every $x\in X$,
\[
T^n x \to x^*,\quad\text{as $n\to\infty$}.
\]
\end{lemma}

\begin{proof}
Fix $x_0\in X$ and define $x_n = T^n x_0$. Then
\[
d(x_{n+1},x_n)
= d(T^{n+1}x_0, T^n x_0)
= d(T^n(Tx_0), T^n x_0)
\le C \tau^n d(Tx_0,x_0)^q.
\]
Thus there exists $A>0$ such that
\[
d(x_{n+1},x_n)\le A \tau^n.
\]
Hence, for $m>n$,
\[
d(x_m,x_n)
\le \sum_{k=n}^{m-1} d(x_{k+1},x_k)
\le A \sum_{k=n}^{\infty} \tau^k
\le \frac{A \tau^n}{1-\tau}.
\]
Therefore $(x_n)$ is Cauchy and converges (by completeness) to some $x^*\in X$.
%
%

Now we show that  $x^*$ is a fixed point.
Fix $n\in\mathbb{N}$. By the triangle inequality,
 \begin{equation}\label{eq.triangular}
d(Tx^*,x^*)
\le d(Tx^*, x_{n+1}) + d(x_{n+1},x^*).
\end{equation}
The second term tends to $0$ as $n\to\infty$, by definition of $x^*$. It remains to show
\[
d(Tx^*, x_{n+1}) \to 0.
\]
%
Since $x_{n+1}=T x_n$, we have
\[
d(Tx^*, x_{n+1}) = d(Tx^*, T x_n).
\]
Now apply the hypothesis with $n=1$:
\[
d(Tx^*, T x_n) \le C\tau \,d(x^*, x_n)^q.
\]
Since $x_n \to x^*$, the right-hand side tends to $0$. Therefore,
$
d(Tx^*, x_{n+1}) \to 0.
$
Combining both terms in~\eqref{eq.triangular} gives
$
d(Tx^*,x^*)=0,
$
and so $Tx^*=x^*$.

%
%
%
Finally, we prove the uniqueness.
If $x^*,y^*$ are fixed points, then for all $n$,
\[
d(x^*,y^*) = d(T^n x^*, T^n y^*) \le C \tau^n d(x^*,y^*)^q.
\]
If $d(x^*,y^*)>0$, dividing by $d(x^*,y^*)^q$ yields
\[
d(x^*,y^*)^{1-q} \le C \tau^n \to 0,
\]
a contradiction. Hence $x^*=y^*$.
\end{proof}

\end{document}